\documentclass[reqno,12pt]{amsart}
\usepackage{amsmath,amssymb,latexsym,soul,cite,mathrsfs}
\usepackage{blindtext}
\usepackage{enumitem}
\usepackage[utf8]{inputenc}
\usepackage{amsthm,amsfonts}
\usepackage{comment}
\usepackage[mathcal]{eucal}
\usepackage{latexsym}
\usepackage[utf8]{inputenc}
\usepackage{bm}
\usepackage[all]{xy}
\usepackage{xcolor}
\usepackage{graphicx}
\usepackage{hyperref}

\usepackage[hyperpageref]{backref}

\newtheorem{theorem}{Theorem}

\newtheorem{theo}{Theorem}[section]

\newtheorem{remark}{Remark}[section]

\newtheorem{lemma}{Lemma}[section]

\newtheorem{prop}{Proposition}[section]

\newtheorem{cor}{Corollary}[section]

\title[Sobolev and Adams-Trudinger-Moser embeddings]{Sharp Sobolev and Adams-Trudinger-Moser embeddings on weighted Sobolev spaces and their applications}

\author[J.M.\ do \'O]{Jo\~ao Marcos do \'O}
\author[G. Lu]{Guozhen Lu}
\author[R.~Ponciano]{Raon\'{\i} Ponciano}

\address[J.M. do \'O]{Dep. Mathematics,
	Federal University of Para\'{\i}ba
	\newline\indent
	58051-900, Jo\~ao Pessoa-PB, Brazil}
\email{\href{mailto:jmbo@pq.cnpq.br}{jmbo@pq.cnpq.br}}

\address[Guozhen Lu]{Dep. Mathematics, University of Connecticut
	\newline\indent
	06269, Storrs-CT, United States of America}
\email{\href{mailto:guozhen.lu@uconn.edu}{guozhen.lu@uconn.edu}}

\address[R.~Ponciano]{Dep. Mathematics,
	Federal University of Para\'{\i}ba
	\newline\indent
	58051-900, Jo\~ao Pessoa-PB, Brazil}
\email{\href{mailto:raoni.cabral.ponciano@academico.ufpb.br}{raoni.cabral.ponciano@academico.ufpb.br}}

\thanks{This work was partially supported by Conselho Nacional de Desenvolvimento Cient\'ifico e Tecnol\'ogico (CNPq) \#312340/2021-4 and \#429285/2016-7, Funda\c c\~ao de Apoio \`a Pesquisa do Estado da Para\'iba (FAPESQ) \#2020/07566-3, Coordena\c c\~ao de Aperfei\c coamento de Pessoal de N\'ivel Superior (CAPES) \#88887.633572/2021-00 and Simons collaboration grants 519099 and 957892 from Simons foundation}

\begin{document}

\begin{abstract}
We derive sharp Sobolev embeddings on a class of Sobolev spaces 
with potential weights without assuming any boundary conditions. 
Moreover, we consider the Adams-type inequalities for the borderline  Sobolev embedding into the exponential class with a sharp constant.
As applications, we prove that the associated elliptic equations with nonlinearities in both forms of polynomial and exponential growths admit nontrivial solutions.
\end{abstract}

\maketitle

\section{Introduction}

It is well-known that Sobolev-type inequalities are essential in studying partial differential equations, especially in studying those arising from geometry and physics. There has been much research on such classes of embeddings and their applications. See, for instance, \cite{MR2424078,MR1422009,MR2838041,MR1617413,MR2777530} and references therein.

Recently in \cite{MR2838041}, motivated by the search for symmetric solutions to boundary value problems, the authors proved sharp pointwise estimates for functions in the Sobolev spaces of radial functions defined in the unit ball $B\subset\mathbb R^N$ centered at the origin where no boundary conditions are assumed. Consequently, they proved that Sobolev spaces of symmetric functions are embedded in specific weighted Lebesgue spaces much higher than the usual  Sobolev exponents when no symmetry is assumed and no weight exists. Thus, it was possible to study boundary value problems in case the nonlinear terms have high polynomial growth. We emphasize that the existence of solutions for supercritical problems was a relevant motivation for this research. These results are summarized as follows:

\begin{theorem}\label{Amora}
Let $k\geq1$ be an integer and let $p\geq1$ be a real number. Let $L^q(B,|x|^\beta)=\{u\colon B\to\mathbb R\mbox{ measurable}\colon\int_B|u|^q|x|^\beta\mathrm dx<\infty\}$ and $W^{k,p}_{\mathrm{rad}}(B)$ is space of functions in $W^{k,p}(B)$ which are radially symmetric.
\begin{flushleft}
If $N>kp$, then $W^{k,p}_{\mathrm{rad}}(B)$ is continuously embedded in $L^q(B,|x|^\beta)$ for every $1\leq q\leq p(N+\beta)/(N-kp)$ and $\beta\geq0$. Moreover, it is compact if $q<p(N+\beta)/(N-kp)$;

If $N=kp$, then $W^{k,p}_{\mathrm{rad}}(B)$ is compactly embedded in $L^q(B,|x|^\beta)$ for every $1\leq q<\infty$ and $\beta\geq0$.
\end{flushleft}
\end{theorem}

The proof of Theorem~\ref{Amora} follows from the following radial estimates:

\begin{itemize}
    \item If $N>kp$, then there exists $C=C(N,k,p)>0$ such that for all $u\in W^{m,p}_{\mathrm{rad}}(B)$
    \begin{equation}
    |u(x)|\leq C\dfrac{\|u\|_{W^{m,p}}}{|x|^{\frac{N-kp}{p}}},\quad\forall x\in \overline B\backslash\{0\};
    \end{equation}
    \item If $N=kp$ and $p>1$, then there exists $C=C(k,p)>0$ such that for all $u\in W^{k,p}_{\mathrm{rad}}(B)$
    \begin{equation}\label{radial-DEO}
        |u(x)|\leq C\|u\|_{W^{k,p}}\left(\left|\log|x|\right|^{\frac{p-1}p}+1\right),\quad\forall x\in \overline B\backslash\{0\}.
    \end{equation}
\end{itemize}

We should mention that motivated by H\'{e}non type equations these kinds of estimates were studied in \cite{MR0674869} for the space $W^{1,2}_{0,\mathrm{rad}}(B)$ and in \cite{MR2396523} for the space $W^{1,2}_{\mathrm{rad}}(B)$. In \cite[Corollary 2.2]{MR2677828}, a similar result to \eqref{radial-DEO} for the space $W^{1,N}_{0,\mathrm{rad}}(B)$ was obtained. See also \cite{MR0454365}, where the author proved a  radial lemma 
for the space $W^{1,2}_{\mathrm{rad}}(\mathbb{R}^N)$ to study Solitary Waves.

Motivated by the previous works, our objective in this paper is to study similar embedding results for a more general class of weighted Sobolev spaces without boundary conditions. In particular, these include sharp Sobolev embedding theorems on higher order weighted Sobolev spaces
without assuming any boundary conditions which extend the non-weighted results in \cite{MR2838041} and improve those results in \cite{MR4112674} with boundary conditions.   At the borderline case, our results improve both those in \cite{MR2838041} by sharpening the embedding to exponential type, namely Adams-Moser-Trudinger type inequalities with best constants,  and the results in \cite{MR2838041} by removing the boundary conditions.

Let $AC(0,R)$ be the space of all absolutely continuous functions on interval $(0,R)$. It is well known that $u\in AC(0,R)$ if and only if $u$ has a derivative $u'$ almost everywhere, which is Lebesgue integrable and $u(r)=u(a)+\int_a^r u'(s) \mathrm ds$. A function $u$ is said to be locally absolutely continuous on $(0,R)$ if for every $r\in (0,R)$ there exists a neighborhood $V_r$ of $r$ such that $u$ is absolutely continuous on $V_r$. Let $AC_{\mathrm{loc}}(0,R)$ be the space of all locally absolutely continuous functions on interval $(0,R)$. 

For each non-negative integer $\ell$ and $0<R\leq\infty$, let $AC_{\mathrm{loc}}^\ell(0,R)$ be the set of all functions $u\colon(0,R)\to \mathbb R$ such that $u^{(\ell)}\in AC_{\mathrm{loc}}(0,R)$, where $u^{(\ell)}=\mathrm d^\ell u/\mathrm dr^\ell$. For $p\geq1$ and $\alpha$ real numbers, we denote by $L^p_\alpha=L^p_\alpha(0,R)$ the weighted Lebesgue 
 space  of measurable functions $u\colon(0,R)\to\mathbb R$ such that
\begin{equation*}
\|u\|_{L^p_\alpha}=\left(\int_0^R|u|^pr^\alpha\mathrm dr\right)^{1/p}<\infty,
\end{equation*}
which is a Banach space under the standard norm $\|u\|_{L_\alpha^p}$. 

For any positive integer $k$ and $(\alpha_0,\ldots,\alpha_k)\in \mathbb R^{k+1}$, with $\alpha_j>-1$ for $j=0,1,\ldots,k$, in\cite{MR4112674} the authors considered the weighted Sobolev spaces for higher-order derivatives $X^{k,p}_{0,R}=X^{k,p}_{0,R}(\alpha_0,\ldots,\alpha_k)$ given by all functions $u\in AC^{k-1}_{\mathrm{loc}}(0,R)$ such that
\begin{equation*}
\lim_{r\to R}u^{(j)}(r)=0,\quad j=0,\ldots,k-1\mbox{ and }u^{(j)}\in L^p_{\alpha_j},\quad j=0,\ldots,k.
\end{equation*}
The space $X^{k,p}_{0,R}$ equipped with the norm
\begin{equation*}
\|u\|_{X^{k,p}_R}=\left(\sum_{j=0}^k\|u^{(j)}\|^p_{L^p_{\alpha_j}}\right)^{1/p}
\end{equation*}
is a Banach space. Supposing $0<R<\infty$ and
\begin{equation}\label{eqweightcondition}
\alpha_{j-1}\geq\alpha_j-p,\quad j=1,\ldots,k,
\end{equation}
we obtain, (see Proposition \ref{prop21JMBO}), that the norms $\|\cdot\|_{X^{k,p}_R}$ and
\begin{equation*}
\|u\|_{X^{k,p}_{0,R}}:=\|u^{(k)}\|_{L^p_{\alpha_k}}=\left(\int_0^R|u^{(k)}|^pr^{\alpha_k}\mathrm dr\right)^{1/p}
\end{equation*}
are equivalent in $X^{k,p}_{0,R}$. Moreover, the following embedding results were established in \cite{MR4112674}:

\begin{theorem}\label{theoB}
Let $p\geq1$ and $0<R<\infty$. Consider $(\alpha_0,\ldots,\alpha_k)\in \mathbb R^{k+1}$, with $\alpha_j>-1$ for $j=0,1,\ldots,k$ and $\theta>-1$ satisfying
\begin{equation}\label{hyphotheoB}
\min\{\theta,\alpha_{j-1}\}\geq\alpha_j-p,\quad \forall j=1,\ldots,k.
\end{equation}
\begin{flushleft}
    If $\alpha_k-kp+1>0$, then the continuous embedding holds:
\begin{equation*}
  X^{k,p}_{0,R}(\alpha_0,\ldots,\alpha_k)\hookrightarrow L^q_\theta,\ \forall  \,  1\leq q\leq p^*:=p^*(\theta,p,k,\alpha_k)=\dfrac{(\theta+1)p}{\alpha_k-kp+1}.  
\end{equation*}
Moreover, it is compactly embedded if $q<p^*$.

  If $\alpha_k-kp+1=0$, then the compact embedding holds:
 \begin{equation*}
  X^{k,p}_{0,R}(\alpha_0,\ldots,\alpha_k)\hookrightarrow L^q_\theta,\quad \forall \,  1\leq q<\infty.   
 \end{equation*}
\end{flushleft}
\end{theorem}

Let us consider, for any positive integer $k$ and $(\alpha_0,\ldots,\alpha_k)\in\mathbb R^{k+1}$, the weighted Sobolev spaces for higher-order derivatives without boundary conditions,
\begin{equation*}
X_R^{k,p}\!=\!X_R^{k,p}(\alpha_0,\ldots,\alpha_k)\!=\!\{u\colon(0,R)\to\mathbb R : u^{(j)}\in L^p_{\alpha_j},\ j=0,1,\ldots,k\}.
\end{equation*}
Using Proposition \ref{prop31} below, one can obtain $u\in AC_{\mathrm{loc}}^{k-1}(0,R)$ for all $u\in X^{k,p}_R$. The spaces $X_R^{k,p}$ and $X^{k,p}_{0,R}$ are complete under the norm
\begin{equation*}
\|u\|_{X_R^{k,p}}=\left(\sum_{j=0}^k\|u^{(j)}\|^p_{L^p_{\alpha_j}}\right)^{1/p}.
\end{equation*}
We can distinguish three special behaviors for the weighted Sobolev spaces $X_R^{k,p}$, namely
\begin{align*}
   & \hspace{-1.5cm} \textbf{Sobolev:} &  \alpha_k-kp+1>0  \\
   & \hspace{-1.5cm} \textbf{Adams-Trudinger-Moser:} &    \alpha_k-kp+1=0  \\
   & \hspace{-1.5cm} \textbf{Morrey:} &  \alpha_k-kp+1<0 
\end{align*}

The space $X_R:=
X_{0,R}^{k,p}$ with $k=1$ was introduced in \cite{MR1422009} to study a  Br\'{e}zis-Nirenberg type problem for a class of quasilinear elliptic operator of the form  $Lu=-r^{-\gamma}(r^\alpha|u'|^\beta u')'$ when considered as acting on radial functions defined on the ball centered at the origin with radius $R$. According to the choice of parameters, the following operators are included in the class: 
\begin{enumerate}
    \item[(i)] $L$ is the Laplacian for $\alpha=\gamma=N-1$ and $\beta=0$;
    \item[(ii)] $L$ is the $p$-Laplacian for $\alpha=\gamma=N-1$ and $\beta=p-2$;
    \item[(ii)] $L$ is the $k$-Hessian for $\alpha=N-k$, $\gamma=N-1$ and $\beta=k-1$.
\end{enumerate}
The suitable space to work on problems like
\begin{equation}
\left\{\begin{array}{ll}
Lu=f(r,u)\mbox{ in }(0,R)&  \\
u'(0)=u(R)=0,&\\
u>0\mbox{ in }(0,R).
\end{array}\right.
\end{equation}
is the space $X_R$ defined as the set of all absolutely continuous function $u\colon(0,R]\to\mathbb R$ such that $u(R)=0$ and
\begin{equation*}
\int_0^R|u'(r)|^{\beta+2}r^\alpha\mathrm dr<\infty.
\end{equation*}
The study on the space $X_R$ is well developed in \cite{MR1422009,MR4112674,MR3209335,MR3575914} and in references therein. But they always suppose $u(R)=0$ in the space $X_R$ to apply a result from \cite{MR1069756} (see Proposition \ref{prop21JMBO}) to guarantee the existence of a constant $C=C(R,q,\gamma,\alpha,\beta)>0$ such that
\begin{equation*}
\left(\int_0^R|u(r)|^{q}r^\gamma\mathrm dr\right)^{\frac1q}\leq C\left(\int_0^R|u'(r)|^{\beta+2}r^\alpha\mathrm dr\right)^{\frac{1}{\beta+2}}
\end{equation*}
for any $1\leq q\leq (\gamma+1)(\beta+2)/(\alpha-\beta-1)$ and $\alpha-\beta-1>0$. Consequently, $X_R$ is continuously embedded in $L^q_\gamma$.

\subsection{Description of the main results}
Now we are in a position to formulate our first main result.

\begin{theo}\label{theo32}
Let $p\geq1$, $0<R<\infty$ and $\theta\geq\alpha_k-kp$.
\begin{flushleft}
If $ \alpha_k-kp+1>0$, then the continuous embedding holds 
\begin{equation*}X^{k,p}_R(\alpha_0,\ldots,\alpha_k)\hookrightarrow L^q_\theta(0,R)\quad \text{for all}\quad 1\leq q\leq p^*:=\dfrac{(\theta+1)p}{\alpha_k-kp+1}. \end{equation*}
Moreover, the embedding is compact if $q<p^*.$ \\
If $\alpha_k-kp+1=0$, then the compact embedding holds
\begin{equation*}
    X^{k,p}_R(\alpha_0,\ldots,\alpha_k)\hookrightarrow L^q_\theta(0,R)\quad \text{for all} \quad q\in[1,\infty).
\end{equation*}
If $\alpha_k-kp+1<0, \; p>1, \; \alpha_k\geq0$
then the continuous embedding holds
\begin{equation*}
   X^{k,p}_R(\alpha_0,\ldots,\alpha_k)\hookrightarrow C^{k-\lfloor\frac{\alpha_k+1}{p}\rfloor-1,\gamma}([0,R]), 
\end{equation*}
\begin{equation*}
    \begin{aligned}
    &  \text{ where }  \gamma=\min\left\{1+\left\lfloor\frac{\alpha_k+1}{p}\right\rfloor-\frac{\alpha_k+1}{p},1-\frac1p\right\}\quad \mbox{if} \quad \frac{\alpha_k+1}{p}\notin\mathbb Z \\
    &     \text{and }\gamma \in (0,1) \quad \mbox{if} \quad \frac{\alpha_k+1}{p}\in\mathbb Z.
    \end{aligned}
\end{equation*}
\end{flushleft}
\end{theo}

The proof of Theorem \ref{theo32} is divided into a sequence of radial lemmas (Propositions \ref{prop32}, \ref{prop33} and \ref{prop34}) in combination with a Hardy type inequality for the space $X_R^{k,p}$ (Proposition \ref{prop35}).

\begin{remark}
In Theorem \ref{theo32}, $p^*=(\theta+1)p/(\alpha_k-kp+1)$ and $\theta\geq\alpha_k-kp$ are sharp. Indeed, consider the space $X_1^{1,1}(1,2)$. For all $q>p^*=1$, $u(t)=1/t^{2/q}$ belongs to $X_1^{1,1}(1,2)\backslash L^q_1$. On other hand, for each $\theta<1=\alpha_k-kp$ define $u(t)=1/t^{\theta+1}$ and note that $u\in X_1^{1,1}(1,2)\backslash L_\theta^{p^*}$.
\end{remark}
\begin{remark}
Theorem \ref{theo32} generalizes Theorem \ref{theoB} in two aspects. Firstly because in Theorem \ref{theo32} it is considered the space $X_R^{k,p}$ instead of $X_{0,R}^{k,p}$ which was the space used in Theorem \ref{theoB}  where  the condition $u^{(j)}(r)\overset{r\to R}\longrightarrow 0$ was necessary in the proof of Theorem \ref{theoB}. Secondly, here it was used the hypothesis $\theta\geq\alpha_k-kp$, which is weaker than the condition \eqref{hyphotheoB} studied in \cite{MR4112674}.
\end{remark}

For the case $\alpha_k-kp+1=0$ from Theorem \ref{theo32} we have  $X^{k,p}_R\hookrightarrow L^q_\theta$  for all $1\leq q<\infty$ and $\theta>-1,$ but one can see that  $X^{k,p}_R \not \hookrightarrow L^{\infty}$  taking $\phi(r)=\log(\log(\frac{eR}r))$.
Then, a natural question arises:
what is the maximal possible growth for a function $g(s)$ such that if $u\in X^{k,p}_R$ implies 
$\int_0^R g(u) \, \mathrm dr < \infty$? Similar to the classical works \cite{MR0960950,MR0216286,MR0301504} and weighted Trudinger-Moser inequalities \cite{DL, DLL, INW, LLZ}, we can conclude that exponential growth is optimal. Precisely, set
\begin{equation}\label{e33}
\ell_{\mu}:= \sup_{\|u\|_{X_R^{k,p}}\leq1}\int_0^Re^{\mu|u|^{\frac{p}{p-1}}}r^\theta \mathrm dr.
\end{equation}
Then we found a $\mu_0>0$ such that \eqref{e33} is finite for $\mu<\mu_0$ and \eqref{e33} is infinite for $\mu>\mu_0$.
Note that we are dealing with a higher-order derivative and the norm $\|u\|_{X^{k,p}_R}$ is different from the one utilized by D.R.~Adams \cite{MR0960950}. 

\begin{theo}\label{theo2}
Let $X^{k,p}_R(\alpha_0,\ldots,\alpha_k)$ with $\alpha_k-kp+1=0$, $p>1$ and $\theta>-1$.  
\begin{flushleft}
$\mathrm{(a)}$ For all $\mu\geq0$ and $u\in X^{k,p}_R,$ we have $\exp(\mu|u|^{p/(p-1)})\in L^1_\theta$.\\
$\mathrm{(b)}$ If $0\leq\mu<\mu_0,$ then $\ell_{\mu}$ is finite, where
\begin{equation*}
\mu_0:=(\theta+1)[(k-1)!]^{p/(p-1)}.
\end{equation*}
Moreover, $\ell_{\mu}$  is attained by a nonnegative function $u_0 \in X^{k,p}_R$ with $\|u_0\|_{X^{k,p}_R}=1$.\\
$\mathrm{(c)}$ If $\mu>\mu_0$ and $\alpha_i-ip+1>0$ for $i=0,\ldots,k-1$, then
$\ell_{\mu}=\infty$.
\end{flushleft}
\end{theo}

%\begin{remark}
%In the proof of item $\mathrm{(b)}$ of Theorem \ref{theo2}, we showed how \eqref{e33} can be bounded for all $R$ (see Proposition \ref{prop51}). If we multiply the integral in \eqref{e33} by $R^{-(\theta+1)}$, then it is bounded by a constant which does not depend on $R$.
%\end{remark}
\begin{remark}
The condition $\alpha_i-ip+1>0$ in item $\mathrm{(c)}$ is natural because if $\alpha_i-ip+1<0$ then the Morrey case of Theorem \ref{theo32} implies that $\ell_\mu$ is finite. For $\alpha_i-ip+1=0$ we can apply the Theorem on $u\in X^{i,p}_R$ because it is in the Adams-Trudinger-Moser case.
\end{remark}

\bigskip

For the case of the first derivative, which corresponds to $k=1$ and $\mu_0=\theta+1,$ we are able to prove that $\ell_{\mu}$ is finite even for the critical case $\mu=\mu_0$ under some additional assumption on the boundary. Precisely, setting
\begin{equation*}
\mathcal K_A=\{u\in X^{1,p}_R\colon\|u\|_{X^{1,p}_R}\leq1\mbox{ and }u(R)\leq Au(r),\ \forall r\in(0,R]\},
\end{equation*}
we have 
\begin{theo}\label{theo3}
Let $\theta>-1$ and $X^{1,p}_R(\alpha_0,p-1)$ with $p\geq2$. For each $A>0$ there exists a constant  $C=C(p,\alpha_0,\theta,A)>0$ such that
\begin{equation*}
\sup_{u\in \mathcal K_A}\int_0^Re^{(\theta+1)|u|^{\frac{p}{p-1}}}r^\theta\mathrm dr\leq CR^{\theta+1}.
\end{equation*}
\end{theo}

\begin{remark}
Note that the condition $u(R)\leq Au(r)$ holds for $A=1$ if we have the Dirichlet boundary condition ($u(R)=0$) or $u$ is nonincreasing.
\end{remark}

\begin{remark}
Theorem 1.1 in \cite{MR3209335} is similar to Theorem \ref{theo3} considering $X^{1,p}_{0,R}$ and the norm of $\|u'\|_{L^p_{\alpha_1}}$ instead of $X^{1,p}_R$ and $\|u\|_{X^{1,p}_R}$.
\end{remark}

A notable consequence of the Theorem \ref{theo2} is the following continuous embedding into weighted Orlicz space
\begin{equation}\label{eq111}
    X^{k,p}_R(\alpha_0,\ldots,\alpha_k)\hookrightarrow L_\Phi(\theta),
\end{equation}
where $\Phi$ is given by
\begin{equation}\label{defphi}
\Phi(t)=\exp(|t|^{p'})
\end{equation}
and $L_\Phi(\theta)$ is the weighted Orlicz space $L_\Phi(\theta):=L_\Phi(0,R;r^\theta\mathrm dr)$ equipped with the Luxemburg norm
\begin{equation*}
\|u\|_{L_\Phi(\theta)}=\inf\left\{\delta>0\colon\int_0^R\Phi(\delta^{-1}u)r^\theta\mathrm dr\leq1\right\}.
\end{equation*}

We also prove that the Sobolev embedding given in \eqref{eq111} is optimal in a natural sense, that is, the space $L_{\Phi}(\theta)$ cannot be replaced by any smaller weighted Orlicz space $L_\Psi(\theta)$. More specifically, if $\Phi$ and $\Psi$ are two $N$-functions, then we say that $\Psi$ increases strictly more rapidly than $\Phi$ and indicate $\Phi\prec\Psi$ if and only if $\Phi(t)/\Psi(\eta t)\overset{t\to\infty}\longrightarrow 0$, for all $\eta>0$. In this sense, we state the following Theorem.
\begin{theo}\label{theo51}
    Suppose \eqref{eqweightcondition}, $\alpha_k-kp+1=0$, and $\alpha_i-ip+1>0$ for all $i=0,\ldots,k-1$. Let $\Psi$ be a N-function such that $\Phi\prec\Psi$ where $\Phi$ is given by \eqref{defphi}. Then $X^{k,p}_{0,R}(\alpha_0,\ldots,\alpha_k)$ cannot be continuously embedded in $L_{\Psi}(\theta)$. In particular, $X^{k,p}_{R}$ cannot be continuously embedded in $L_\Psi(\theta)$ either.
\end{theo}

Let $X^{k,p}_{\mathcal N,\gamma,R}(\alpha_0,\ldots,\alpha_k)$ be the weighted Sobolev space with Navier boundary condition given by
\begin{equation*}
\left\{u\in X_R^{k,p}(\alpha_0,\ldots,\alpha_k)\colon \Delta_{\gamma}^j u(R)=0\ \forall j=0,\ldots,\left\lfloor\frac{k-1}2\right\rfloor\right\},
\end{equation*}
where $\Delta_{\gamma}u=-r^{-\gamma}(r^{\gamma}u')'=-u''-\gamma u'/r$. The $\gamma$-generalized $k$th order gradient of $u$, denoted by $\nabla_{\gamma}^ku$, is defined to be
\begin{equation*}
\nabla_{\gamma}^ku=\left\{
\begin{array}{ll}
    \left(\Delta_{\gamma}^{\frac{k-1}2}u\right)'&\mbox{for }k\mbox{ odd,}\\
    \Delta_{\gamma}^{\frac k2}u&\mbox{for }k\mbox{ even.} 
\end{array}\right.
\end{equation*}
Under natural assumptions on the weights and $\gamma$, we are able to show that the norms $\|\nabla_\gamma^k\cdot\|_{L^p_{\alpha_k}}$ and $\|\cdot\|_{X^{k,p}_{R}}$ are equivalents in $X^{k,p}_{\mathcal N,\gamma,R}$ (see Proposition \ref{propequivnormkgrad}). The following Theorem is a generalization to the classical result from Adams' inequality \cite{MR0960950} for radially symmetric functions.

\begin{theo}\label{theoainbc}
Let $X^{k,p}_{\mathcal N,\gamma,R}(\alpha_0,\ldots,\alpha_k)$ such that $\alpha_k-kp+1=0$. If $\theta>-1$ and $\gamma>k-1$ for $k$ even and $\gamma>k-2$ for $k$ odd, then
\begin{equation*}
\sup_{R\in (0,\infty)}\sup_{\underset{\|\nabla^k_{\gamma}u\|_{L^p_{\alpha_k}}=1}{{u\in X^{k,p}_{\mathcal N,\gamma,R}}}}R^{-(\theta+1)}\int_0^Re^{\mu|u|^{p'}}r^\theta\mathrm dr<\infty\Leftrightarrow \mu\leq\mu_0,
\end{equation*}
where
\begin{equation*}
\mu_0=\mu_0(\theta,\gamma,p,k):=\left\{\begin{array}{ll}
     (\theta+1)\left(2^{k-1}\dfrac{\Gamma(\frac{k+1}2)\Gamma(\frac{\gamma+1}2)}{\Gamma(\frac{\gamma+2-k}2)}\right)^{\frac{p}{p-1}},&k\mbox{ odd},\\
     (\theta+1)\left(2^{k-1}\dfrac{\Gamma(\frac{k}2)\Gamma(\frac{\gamma+1}2)}{\Gamma(\frac{\gamma+1-k}2)}\right)^{\frac{p}{p-1}},&k\mbox{ even}.
\end{array}\right.
\end{equation*}
\end{theo}

\begin{remark}
In the particular case of $\gamma=\alpha_k$ and under stronger Dirichlet boundary condition, a similar result was proved in \cite[Theorem 1.3]{MR4112674}. However, the proof given there only shows that the above supremum is infinite for $\mu_0=\mu_0(\theta, \alpha_k, k)$ and is finite for
$$
\mu\leq (\theta+1)\left(2^{k-1}(p-1)\dfrac{\Gamma(\frac{k}2)\Gamma(\frac{\alpha_k-1}2)}{\Gamma(\frac{\alpha_k+1-k}2)}\right)^{p'}$$
which is strictly smaller than 
$\mu_0=\mu_0(\theta, \alpha_k,k)$ except in the case of $k=1$.

In fact, in \cite{MR4112674} (see page 535),   it was claimed that 
$$
\mu_0(\alpha,\theta,k)=C_{k,\alpha}^{-p'}\mu_0(\alpha-2(l-1)p,\theta,2).
$$
However, this identity does not hold and by a careful calculation 
the following identity holds:
$$
\mu_0(\alpha,\theta,k)=C_{k,\alpha}^{-p'}\mu_0(\alpha,\theta,2).
$$
 
\end{remark}

\subsection{Application}

As an application of the embedding given by Theorem \ref{theo32}, we study the fourth order problem
\begin{equation}\label{problems}
\left\{\begin{array}{ll}
     \Delta_\alpha^2u=r^{\theta-\alpha}g(r)|u|^{p-2}u&\mbox{in }(0,R),  \\
     u=\Delta_\alpha u=0&\mbox{in }R,\\
     u'=(\Delta_\alpha u)'=0&\mbox{in }0,
\end{array}\right.
\end{equation} 
where $g\colon[0,R]\to\mathbb R$ is continuous positive, $p$ is subcritical and $\Delta_\alpha u=-r^{-\alpha}(r^\alpha u')'$ is the $\alpha$-generalized radial Laplace operator. When $\alpha>3$, the function space associated with the problem \eqref{problems} fits in the Sobolev case of the Theorem \ref{theo32}. More specifically, we prove the existence and regularity results of weak solutions for \eqref{problems}.

\begin{theo}\label{theo4}
Suppose $2\leq p<2(\theta+1)/(\alpha-3)$ and $\theta>\alpha-1$. Then there exists $u_0\in C^4((0,R])\cap C^3([0,R])$ a nontrivial classical solution of \eqref{problems} with $\Delta_\alpha u_0\in C^2((0,R])\cap C^1([0,R])$. Moreover, $u_0''(0)=-\Delta_\alpha u_0(0)/(\alpha+1)$ and $u_0'''(0)=0$.
\end{theo}

We also study an application of the attainability given by Theorem \ref{theo2} in the problem
\begin{equation}\label{problemtm}
\left\{\begin{array}{ll}
\Delta_3 ^2u=r^{\theta-3}f(r,u)&\mbox{in }(0,R),\\
u=\Delta_3 u=0&\mbox{in }R,\\
u'=(\Delta_3 u)'=0&\mbox{in }0,
\end{array}\right.
\end{equation}
where $f$ is subcritical (that is $f$ satisfies \eqref{eqfTM} with $\mu<\theta+1$). More specifically, we prove the existence of weak solutions and develop the regularity theory for those problems.

\begin{theo}\label{theo5}
Suppose $\theta>2$ and $f$ satisfying \eqref{eqfTM} with $f(r,\cdot)$ an odd function and $f(r,t)\geq0$ for all $t\geq0$. If $\mu<\theta+1$, then there exists $u_0\in C^4((0,R])\cap C^3([0,R])$ a nontrivial classical solution of
\begin{equation*}
\left\{\begin{array}{ll}
\Delta_3 ^2u=r^{\theta-3}\lambda f(r,u)&\mbox{in }(0,R),\\
u=\Delta_3 u=0&\mbox{in }R,\\
u'=(\Delta_3 u)'=0&\mbox{in }0,
\end{array}\right.
\end{equation*}
with $\lambda=\int_0^Rf(r,u_0)u_0r^\theta\mathrm dr$. Moreover, $\Delta_3 u_0\in C^2((0,R])\cap C^1([0,R])$, $u_0''(0)=-\Delta_3u_0(0)/4$ and $u_0'''(0)=0$.
\end{theo}

\subsection{Notation} We use the following notation. 

\begin{itemize}
    \item $\lfloor x\rfloor=\max\{n\in\mathbb Z\colon n\leq x\}$ for each $x\in\mathbb R$.
    \item $p'=\frac{p}{p-1}$ for each $1<p<\infty$.
\end{itemize}

\subsection{Organization of the paper} 
In Sect.~\ref{RL}, we obtain some Radial Lemmas for $X^{k,p}_R$ and prove Theorem \ref{theo32}. 
In Sect.~\ref{MTH}, we work in the Adams-Trudinger-Moser case ($\alpha_k-kp+1=0$) finding some conditions such that the supremum \eqref{e33} is finite or infinite. More specifically, we prove Theorems \ref{theo2}, \ref{theo3} and \ref{theo51}. Sect. \ref{AIN} will give Adams' inequality for weighted Sobolev spaces (Theorem \ref{theoainbc}). Sect.~\ref{APDE} will provide a study on the PDE's \eqref{problems} and \eqref{problemtm} proving Theorems \ref{theo4} and \ref{theo5}. 
Moreover, this Section develops regularity theory on the PDE: $\Delta_\alpha^2u=r^{\theta-\alpha}f(r,u)$.

\section{Radial Lemmas and proof of Theorem \ref{theo32}}\label{RL}

We begin by recalling some definitions following the notation of \cite{MR1069756}. We say that a function $u\in AC_{\mathrm{loc}}(0,R)$ belongs to class $AC_L(0,R)$ if $\lim_{r\to0}u(r)=0$. Analogously, $u\in AC_{\mathrm{loc}}(0,R)$ belongs to $AC_R(0,R)$ if $\lim_{r\to R}u(r)=0$. In \cite[Example 6.9]{MR1069756} was proved the following Hardy-type inequality:
\begin{prop}\label{prop21JMBO}
Given $p,q\in[1,\infty)$ and $\theta,\alpha\in\mathbb R$ the inequality
\begin{equation*}
\left(\int_0^R|u|^qr^\theta\mathrm dr\right)^{\frac{1}{q}}\leq C\left(\int_0^R|u'|^pr^\alpha\mathrm dr\right)^{\frac{1}{p}}
\end{equation*}
holds for some constant $C=C(p,q,\theta,\alpha,R)>0$ under the following conditions:
\begin{flushleft}
    $\mathrm{(i)}$ for $u\in AC_L(0,R)$ if and only if one of the following two conditions are fulfilled:\\
        $\mathrm{(1)}$ $1\leq p\leq q<\infty$, $q\geq\frac{(\theta+1)p}{\alpha-p+1}$, and $\alpha-p+1<0$.\\
        $\mathrm{(2)}$ $1\leq q<p<\infty$, $q>\frac{(\theta+1)p}{\alpha-p+1}$, and $\alpha-p+1<0$;\\
    $\mathrm{(ii)}$ for $u\in AC_R(0,R)$ if and only if one of the following two conditions is fulfilled:\\
        $\mathrm{(1)}$ If $1\leq p\leq q<\infty$, then
        \begin{equation*}
        1\leq p\leq q\leq\dfrac{(\theta+1)p}{\alpha-p+1}\mbox{ and }\alpha-p+1>0,
        \end{equation*}
        or
        \begin{equation*}
        \theta>-1\mbox{ and }\alpha-p+1\leq0.
        \end{equation*}
        $\mathrm{(2)}$ If $1\leq q<p$, then
        \begin{equation*}
        1\leq q<p<\infty,\mbox{ with }q<\dfrac{(\theta+1)p}{\alpha-p+1}\mbox{ and }\alpha-p+1>0,
        \end{equation*}
        or
        \begin{equation*}
        \theta>-1\mbox{ and }\alpha-p+1\leq0.
        \end{equation*}
\end{flushleft}
\end{prop}

%\begin{cor}
%Let $p\geq1$ and $\alpha_j\geq \alpha_k-(k-j)p$ for all $j=0,\ldots,k$. Then the norms $\|\cdot\|_{X^{k,p}_{0,R}}$ and $\|\cdot\|_{X^{k,p}_{R}}$ are equivalent in the space $X^{k,p}_{0,R}$.
%\end{cor}

As in the classical Sobolev spaces, we have the continuous injection  $W^{k,p}(I)\hookrightarrow AC^{k-1}(I)$ (for $I$ a bounded interval), we also have the following similar result  for the weighted Sobolev space $X^{k,p}_R$:

\begin{prop}\label{prop31}
Let $u\in X_R^{k,p}(\alpha_0,\ldots,\alpha_k)$ with $1\leq p < \infty$. Then there exists  $U\in AC_{\mathrm{loc}}^{k-1}((0,R])$ such that 
\begin{equation*}
    u = U \quad \text{a.e. on } \quad (0,R). 
\end{equation*}
Moreover, $U^{(k)}$ (in the classical sense) exists a.e. on $ (0,R)$, $U^{(k)}$ is a measurable function and
\begin{equation*}
\int_0^R\left|U^{(j)}(t)\right|^pt^{\alpha_j}\mathrm dt<\infty\mbox{ for }j=0,1,\ldots,k.
\end{equation*}
\end{prop}
\begin{proof}
Let $u\in X_R^{1,p}(\alpha_0,\alpha_1)$ and $n_0\in\mathbb N$ such that $1/n_0<R$. Note that $u\in W^{1,p}(1/n,R)$ for all $n\geq n_0,$ because $1\leq t^{\alpha_i}/\min_{t\in [1/n,R]}t^{\alpha_i}$ for all $t\in[1/n,R]$ and $n\geq n_0$. 
By  \cite[Theorem 8.2]{MR2759829}, for each $n\geq n_0$ there exists $U_n\in AC([1/n,R])$ such that $U_n=u$ a.e. in $(1/n,R)$. 
The proof for the case $k=1$ is completed by defining $U(t):=U_n(t)$ for each $t\in (0,R]$ where $n$ is a natural with $1/n<t$. 
For $u\in X_R^{k,p}(\alpha_0,\ldots,\alpha_k)$, just use the case $k=1$ on $u^{(k-1)}\in X_R^{1,p}(\alpha_{k-1},\alpha_k)$ (see Remark 6 after \cite[Theorem 8.2]{MR2759829}).
\end{proof}

\begin{remark}
The conversely of Propostion~\ref{prop31} also holds, and the proof is trivial. 
\end{remark}

Since $u\in X^{1,p}_R$ does not necessarily satisfy $u(R)=0$, we need the next Lemma to obtain our radial lemmas. 

\begin{lemma}\label{lemma31}
There exists a constant $C=C(\alpha_0,\alpha_1,p,R)>0$ such that
\begin{equation*}
|u(R)|\leq C \|u\|_{X_R^{1,p}}, \quad \forall u\in X_R^{1,p}(\alpha_0,\alpha_1).
\end{equation*}
\end{lemma}
\begin{proof}
By Mean Value Theorem for Integrals and Proposition \ref{prop31}, there  exists $t_0\in (R/2,R)$ such that
\begin{equation}\label{eq12}
u(t_0)=\frac2R\int_{\frac R2}^Ru(r)\mathrm dr.
\end{equation}
On other hand,
\begin{equation}\label{eq13}
u(R)-u(t_0)=\int_{t_0}^Ru'(r)\mathrm dr.
\end{equation}
Then, from \eqref{eq12} and \eqref{eq13},
\begin{align*}
|u(R)|&\leq |u(R)-u(t_0)|+|u(t_0)|\\
&\leq\int_{\frac R2}^R\left|u'(r)\right|r^{\frac{\alpha_1}p}\cdot r^{-\frac{\alpha_1}{p}}\mathrm dr+\frac2R\int_{\frac R2}^R|u(r)|r^{\frac{\alpha_0}p}\cdot r^{-\frac{\alpha_0}p}\mathrm dr\\
&\leq \|u'\|_{L^p_{\alpha_1}(0,R)}\left\|r^{-\frac{\alpha_1}p}\right\|_{L^{p'}(\frac{R}{2},R)}+\frac2R\|u\|_{L^p_{\alpha_0}(0,R)}\left\|r^{-\frac{\alpha_0}p}\right\|_{L^{p'}(\frac R2,R)}\\
&\leq C\|u\|_{X_R^{1,p}}.
\end{align*}
\end{proof}

Now we are ready to prove the Radial Lemma for the Sobolev case ($\alpha_k-kp+1>0$) and for the Adams-Trudinger-Moser case ($\alpha_k-kp+1=0$). The Proposition \ref{prop32} is for the first case and for the second case we need to split it into Proposition \ref{prop33} (when $p>1$) and Proposition \ref{prop34} (when $p=1$).

\begin{prop}\label{prop32}
Suppose $\alpha_k-kp+1>0$. Then, there exists a constant $C=C(\alpha_0,\ldots,\alpha_k,p,k,R)>0$ such that for all $u\in X_R^{k,p}(\alpha_0,\ldots,\alpha_k)$ it holds 
\begin{equation}\label{LR1}
|u(t)|\leq C\dfrac1{t^{\frac{\alpha_k-kp+1}{p}}}\|u\|_{X_R^{k,p}},\quad\forall t\in(0,R].
\end{equation}
\end{prop}
\begin{proof}
We are going to prove by induction on $k$.
Let $u\in X_R^{1,p}$ and $t\in(0,R]$. As \cite{MR2838041}, the identity
\begin{equation*}
u(t)=u(R)-\int_t^Ru'(r)r^{\frac{\alpha_1}{p}}\cdot r^{-\frac{\alpha_1}{p}}\mathrm dr
\end{equation*}
holds. From H\"older's inequality and Lemma \ref{lemma31}, we get
\begin{equation*}
|u(t)|\leq C\dfrac{1}{t^{\frac{\alpha_1-p+1}p}}\|u\|_{X_R^{1,p}},
\end{equation*}
which proves \eqref{LR1} for $k=1$. Now, let $\alpha_k-kp+1>0$, $u\in X_R^{k,p}$ and $t\in(0,R]$. From $u'\in X_R^{k-1,p}(\alpha_1,\ldots,\alpha_k)$ and the induction hypothesis, we have
\begin{equation*}
|u(t)|\leq|u(R)|+\int_t^R|u'(r)|\mathrm dr\leq |u(R)|+C\|u'\|_{X_R^{k-1,p}}\int_t^R\dfrac1{r^{\frac{\alpha_k-(k-1)p+1}{p}}}\mathrm dr.
\end{equation*}
Therefore, by Lemma \ref{lemma31}, we conclude that \eqref{LR1} holds.
\end{proof}

\begin{prop}\label{prop33}
If $\alpha_k-kp+1=0$ and $p>1$, 
then there exists a constant $C=C(\alpha_0,\ldots,\alpha_{k-1},p,k,R)>0$ such that for all $u\in X_R^{k,p}$ it holds 
\begin{equation}\label{LR-A}
|u(t)|\leq |\log (t/R)|^{\frac{p-1}p}\|u^{(k)}\|_{L^p_{\alpha_k}}+C\|u\|_{X^{k,p}_R},\quad\forall t\in(0,R].
\end{equation}
If $u\in X^{1,p}_{0,R}(\alpha_0,p-1)$, then \eqref{LR-A} holds with $C=0$.
\end{prop}
\begin{proof}
Using integration by parts ($k-1$)-times we have the identity
\begin{equation*}
\int_t^Ru^{(k)}(r)r^{k-1}\mathrm dr=(-1)^k(k-1)!\sum_{j=0}^{k-1}\dfrac{(-1)^j}{j!}\left(u^{(j)}(t)t^j-u^{(j)}(R)R^j\right).
\end{equation*}
Then
\begin{equation}
    u(t)=\dfrac{(-1)^k}{(k-1)!}\int_t^Ru^{(k)}(r)r^{k-1}\mathrm dr+u(R)-\sum_{j=1}^{k-1}\dfrac{(-1)^j}{j!}\left(u^{(j)}(t)t^j-u^{(j)}(R)R^j\right)].\label{eq15}
\end{equation}
Since $\alpha_k-kp+1=0$,
\begin{equation}\label{eq16}
\int_t^Ru^{(k)}(r)r^{k-1}\mathrm dr\leq \left|\log(t/R)\right|^{\frac{p-1}p}\|u^{(k)}\|_{L^p_{\alpha_k}}.
\end{equation}
Proposition \ref{prop32}, applied to $u^{(j)}\in X_R^{k-j,p}(\alpha_j,\ldots,\alpha_k)$, ensures that
\begin{equation}\label{eq17}
\sum_{j=1}^{k-1}\dfrac{(-1)^j}{j!}u^{(j)}(t)t^jds\leq C\|u\|_{X^{k,p}_R}.
\end{equation}
Therefore, \eqref{LR-A} is a consequence of \eqref{eq15}, \eqref{eq16}, \eqref{eq17}, and Lemma \ref{lemma31}.
\end{proof}
\begin{prop}\label{prop34}
If $\alpha_k-kp+1=0$ and $p=1$, then $X^{k,p}_R\hookrightarrow C([0,R])$. In other words, $X_R^{k,1}(\alpha_0,\ldots,\alpha_{k-1},k-1)\hookrightarrow C([0,R])$.
\end{prop}
\begin{proof}
We note that $u^{(j)}\in X_R^{k-j,1}(\alpha_j,\ldots,\alpha_k)$ for all $j=1,\ldots,k-1$. Using Proposition \ref{prop32} it follows
\begin{equation*}
|u^{(j)}(t)|\leq C\dfrac1{t^j}\|u\|_{X_R^{k,1}}.
\end{equation*}
In view of \eqref{eq15}, Lemma \ref{lemma31}, and $u^{(k)}(t)t^{k-1}\in L^1$, we have
\begin{equation*}
|u(t)|\leq C\|u\|_{X^{k,1}_R},\quad\forall t\in(0,R].
\end{equation*}

It remains to prove that $\lim_{t\to0}u(t)$ exist. Since $u^{(k)}(t)t^{k-1}\in L^1$ and \eqref{eq15}, it is enough to show that there exist $\lim_{t\to0}u^{(j)}(t)t^j$ for all $j=1,\ldots,k-1$.  Let $1\leq i\leq k-1$ integer. Using integration by parts, we have
\begin{equation}\label{eq18}
\int_t^Ru^{(i+1)}(r)r^{i}\mathrm dr=u^{(i)}(R)R^i-u^{(i)}(t)t^i-i\int_t^Ru^{(i)}(r)r^{i-1}\mathrm dr.
\end{equation}
If $u^{(i)}\in L^1_{i-1}$ for all $i=1,\ldots,k$, we can apply limit with $t\to0$ in \eqref{eq18} to conclude the Proposition \ref{prop34}. In order to prove $u^{(k-i)}\in L^1_{k-i-1}$ for all $i=0,\ldots,k-1$ we proceed by induction on $i$. The case $i=0$ follows since $u\in X^{k,1}_R(\alpha_0,\ldots,k-1)$. Suppose $i<k-1$ and $u^{(k-i)}\in X^{i,1}_R(\alpha_{k-i},\ldots,k-1)$. By Proposition \ref{prop21JMBO} and Lemma \ref{lemma31}
\begin{align*}
\int_0^R|u^{(k-i-1)}(r)|r^{k-i-2}\mathrm dr&\leq C\|u\|_{X_R^{k,p}}+\int_0^R\left|u^{(k-i-1)}(r)-u^{(k-i-1)}(R)\right|r^{k-i-2}\mathrm dr\\
&\leq C\|u\|_{X_R^{k,p}}+C\int_0^R|u^{(k-i)}(r)|r^{k-i-1}\mathrm dr<\infty,
\end{align*}
which completes the proof of Proposition \ref{prop34}.
\end{proof}
Before the proof of the Hardy-type inequality given by Proposition \ref{prop35}, we are going to state the following result proved in  \cite[Theorem 4.3 and Remark 4.4]{MR1982932}.
\begin{theo}\label{theo31}
Let $1<p<\infty$. Then the inequality
\begin{equation}\label{eq19}
\int_0^1|z(t)|^pu(t)\mathrm dt\leq C\int_0^1\left|z^{(m)}(t)\right|^pv(t)\mathrm dt
\end{equation}
holds for all $z\in W^{m,p}((0,1),v)$ such that $z^{(j)}(1)=0$, for any $j=0,1,\ldots,m-1$, if, and only if,
\begin{equation}\label{eq20}
\left.
\begin{aligned}
    \sup_{0<x<1}\left(\displaystyle\int_0^x(x-t)^{(m-1)p}u(t)\mathrm dt\right)^{1/p}\left(\displaystyle\int_x^1v(t)^{-\frac{1}{p-1}}\mathrm dt\right)^{\frac{p-1}p}<\infty\\
\displaystyle\sup_{0<x<1}\left(\displaystyle\int_0^xu(t)\mathrm dt\right)^{1/p}\left(\displaystyle\int_x^1(t-x)^{\frac{(m-1)p}{p-1}}v(t)^{-\frac1{p-1}}\mathrm dt\right)^{\frac{p-1}p}<\infty
\end{aligned}
\right\}
\end{equation}
where $u$ and $v$ are positive measurable function in $(0,1)$.
\end{theo}
\begin{prop}\label{prop35}
Given $j=0,1,\ldots,k$ with $\alpha_k-jp+1>0$, then there exists a constant $C_j=C(j,p,R,k,\alpha_k)>0$ such that for all $u\in X_R^{k,p},$
\begin{equation}\label{eq21}
\int_0^R\left|\dfrac{u^{(k-j)}(t)}{t^j}\right|^pt^{\alpha_k}\mathrm dt\leq C_j\sum_{i=k-j}^k\int_0^R\left|u^{(i)}(t)\right|^pt^{\alpha_i}\mathrm dt.
\end{equation}
In particular,
\begin{equation*}
\left\|\dfrac{u}{t^k}\right\|_{L^p_{\alpha_k}}\leq C_k\|u\|_{X^{k,p}_R}.
\end{equation*}
\end{prop}
\begin{proof}
Since $u^{(k-j)}\in X^{j,p}_R(\alpha_{k-j},\ldots,\alpha_k)$ for all $u\in X_R^{k,p}(\alpha_0,\ldots,\alpha_k)$, it is enough to prove \eqref{eq21} for $j=k$, i.e.
\begin{equation}\label{eq22}
\int_0^R\left|\dfrac{u(t)}{t^k}\right|^pt^{\alpha_k}\mathrm dt\leq C\sum_{i=0}^k\int_0^R\left|u^{(i)}(t)\right|^pt^{\alpha_i}\mathrm dt.
\end{equation}
Given $u\in X^{k,p}_R,$ we define $v\in X_R^{k,p}$ by
\begin{equation}\label{eq22.1}
v(t)=u(t)-\sum_{i=0}^{k-1}\dfrac{u^{(i)}(R)}{i!}(t-R)^i.
\end{equation}
Consequently, $v^{(i)}(R)=0$ for $i=0,\ldots,k-1$ and $v^{(k)}=u^{(k)}$. Thus $v\in X^{k,p}_{0,R}$.

\noindent {\bf Case $p>1$:}
Let us first prove \eqref{eq20} holds for $u(t)=t^{\alpha_k-kp}$ and $v(t)=t^{\alpha_k}$. Indeed,

\begin{align*}
&\left(\displaystyle\int_0^x(x-t)^{(k-1)p}t^{\alpha_k-kp}\mathrm dt\right)^{1/p}\left(\displaystyle\int_x^Rt^{-\frac{\alpha_k}{p-1}}\mathrm dt\right)^{\frac{p-1}p}\\
&\quad\leq \left(x^{(k-1)p}\int_0^xt^{\alpha_k-kp}\mathrm dt\right)^\frac{1}{p}C\left(\dfrac1{x^{\frac{\alpha_k}{p-1}-1}}-\dfrac{1}{R^{\frac{\alpha_k}{p-1}-1}}\right)^{\frac{p-1}p}\\
&\quad\leq C\left(x^{(k-1)p+\alpha_k-kp+1}\right)^{1/p}x^{-\frac{\alpha_k-p+1}{p}}
\leq C
\end{align*}
and for all $0<x<R$,
\begin{align*}
&\left(\displaystyle\int_0^xt^{\alpha_k-kp}\mathrm dt\right)^{1/p}\left(\displaystyle\int_x^R(t-x)^{\frac{(k-1)p}{p-1}}t^{-\frac{\alpha_k}{p-1}}\mathrm dt\right)^{\frac{p-1}p}\\
&\quad\leq Cx^{\frac{\alpha_k-kp+1}{p}}\left(\int_x^Rt^{-\frac{\alpha_k-kp+p}{p-1}}\mathrm dt\right)^{\frac{p-1}p}\\
&\quad\leq Cx^{\frac{\alpha_k-kp+1}p}\left(\dfrac{1}{x^{-\frac{\alpha_k-kp+1}{p-1}}}-\dfrac1{R^{-\frac{\alpha_k-kp+1}{p-1}}}\right)^{\frac{p-1}p}
\leq C .
\end{align*}
By Theorem \ref{theo31},
\begin{equation*}
\int_0^R\left|\dfrac{v(t)}{t^k}\right|^pt^{\alpha_k}\mathrm dt\leq C \int_0^R\left|v^{(k)}(t)\right|^pt^{\alpha_k}\mathrm dt.
\end{equation*}
Also,  by Lemma \ref{lemma31} and \eqref{eq22.1}, we have $|u(t)|\leq |v(t)|+C\|u\|_{X^{k,p}_R}$. Therefore
\begin{equation*}
\int_0^R\left|\dfrac{u(t)}{t^k}\right|^pt^{\alpha_k}\mathrm dt\leq C\int_0^R\left|u^{(k)}(t)\right|^pt^{\alpha_k}\mathrm dt+C\|u\|^p_{X^{k,p}},
\end{equation*}
which gives \eqref{eq22}.

\noindent {\bf Case $p=1$:} Since $|u(t)|\leq |v(t)|+C\|u\|_{X^{k,p}_R}$, it is sufficient to prove that
\begin{equation}\label{eq23}
\int_0^R|v(t)|t^{\alpha_k-k}\mathrm dt\leq C\int_0^R\left|v^{(k)}(t)\right|t^{\alpha_k}\mathrm dt.
\end{equation}
Integrating by parts $k$ times the left-hand side in \eqref{eq23}, we obtain \eqref{eq23}.
\end{proof}

%\begin{proof}
%Using Proposition \ref{prop32}, we have $C>0$ such that
%\begin{equation*}
%|u(r)|^{\frac{p\theta}{\alpha_k-kp+1}}r^\theta\leq C\|u\|_{X^{k,p}}^{\frac{p\theta}{\alpha_k-kp+1}}\quad \forall r\in (0,R).
%\end{equation*}
%Then
%\begin{align*}
%\int_0^R|u(r)|^{\frac{(\theta+1)p}{\alpha_k-kp+1}}r^\theta\mathrm dr&\leq C\|u\|_{X^{k,p}}^{\frac{p\theta}{\alpha_k-kp+1}}\int_0^R|u(r)|^{\frac{p}{\alpha_k-kp+1}}\mathrm dr\\
%&\leq ???
%\end{align*}
%\textcolor{red}{Se tivermos $X^{k,p}\hookrightarrow L^{\frac{p}{\alpha_k-kp+1}}$ o resultado segue.}
%\end{proof}

Therefore, we have done all the necessary steps to prove the Theorem \ref{theo32} for the Sobolev and Adams-Trudinger-Moser case. 

We will prove two preliminary results to prove the Morrey case in Theorem \ref{theo32}.  We follow the same argument used in the classical Sobolev space $W^{k,p}(\Omega)$ ($\Omega\subset\mathbb R^N$) when $kp>N$.

\begin{lemma}\label{lemmamorrey}
Let $p>1$, $0<R\leq\infty$, $\alpha_0\in\mathbb R$ and $\alpha_1\geq0$. If $\alpha_1-p+1<0$, then for all $u\in X^{1,p}_R(\alpha_0,\alpha_1)$, we have
\begin{equation*}
|u(r)-u(s)|\leq C|r-s|^\gamma\|u'\|_{L^p_{\alpha_1}}\quad\forall r,s\in (0,R),
\end{equation*}
where $\gamma=1-(\alpha_1+1)/p$ and $C=C(p,\alpha_1)>0$.
\end{lemma}

\begin{proof}
Given $r,s\in(0,R)$ with $r>s$, define $I=(r,s)$. For $x,y\in \overline I$ such that $x>y$ we have
\begin{align}
|u(x)-u(y)|&\leq |x-y|\int_0^1|u'(y+t(x-y))|\mathrm dt\nonumber\\
&\leq |r-s|\int_0^1|u'(y+t(x-y))|\mathrm dt.\label{morrey1}
\end{align}
Set
\begin{equation*}
\overline u=\dfrac1{|r-s|}\int_{I}u(x)\mathrm dx.
\end{equation*}
Integrating \eqref{morrey1} on $x\in I$ and changing variables $z=h(x)=y+t(x-y)$, we get
\begin{align}
|\overline u-u(y)|&\leq\int_I\int_0^1|u'(y+t(x-y))|\mathrm dt\mathrm dx\nonumber\\
&=\int_0^1\int_I|u'(y+t(x-y))|\mathrm dx\mathrm dt\nonumber\\
&=\int_0^1\int_{h(I)}|u'(z)|\mathrm dz\dfrac{\mathrm dt}{t}.\label{morrey2}
\end{align}
Since $h(I)\subset I$,
\begin{align*}
\int_{h(I)}|u'(z)|\mathrm dz&\leq \|u'\|_{L^p_{\alpha_1}(I)}\left(\int_{h(I)}z^{-\frac{\alpha_1}{p-1}}\mathrm dz\right)^{\frac{p-1}p}\\
&=\left(\dfrac{p-1}{p-1-\alpha_1}\right)^{\frac{p-1}p}\|u'\|_{L^p_{\alpha_1}}\!\left[h(r)^{\frac{p-1-\alpha_1}{p-1}}-h(s)^{\frac{p-1-\alpha_1}{p-1}}\right]^{\frac{p-1}p}\\
&\leq\left(\dfrac{p-1}{p-1-\alpha_1}\right)^{\frac{p-1}p}\|u'\|_{L^p_{\alpha_1}}t^{\frac{p-1-\alpha_1}{p}}|r-s|^{1-\frac{\alpha_1+1}p},
\end{align*}
where the last inequality we used $\alpha_1\geq0$ and $a^\beta-b^\beta\leq(a-b)^\beta$ for all $a\geq b\geq0$ and $0<\beta\leq1$. Using \eqref{morrey2} we have
\begin{align*}
|\overline u-u(y)|&\leq\left(\dfrac{p-1}{p-1-\alpha_1}\right)^{\frac{p-1}p}\|u'\|_{L^p_{\alpha_1}(I)}\int_0^1t^{-\frac{\alpha_1+1}p}\mathrm dt|r-s|^{1-\frac{\alpha_1+1}p}\\
&=\left(\dfrac{p-1}{p-1-\alpha_1}\right)^{\frac{p-1}p}\dfrac{p}{p-1-\alpha_1}\|u'\|_{L^p_{\alpha_1}(I)}|r-s|^{1-\frac{\alpha_1+1}p}.
\end{align*}
Adding this inequality for $y=r$ and $y=s$, we obtain
\begin{equation*}
|u(r)-u(s)|\leq\left(\dfrac{p-1}{p-1-\alpha_1}\right)^{\frac{p-1}p}\dfrac{2p}{p-1-\alpha_1}\|u'\|_{L^p_{\alpha_1}(I)}|r-s|^{1-\frac{\alpha_1+1}p}.
\end{equation*}
\end{proof}

Note that our succeeding Theorem is exactly the Morrey case of Theorem \ref{theo32} when $k=1$.

\begin{theo}\label{theomorrey}
Let $p>1$, $\alpha_0,\alpha_1\in\mathbb R$ and $0<R<\infty$. If $\alpha_1-p+1<0$, then following continuous embedding holds
\begin{equation*}
X^{1,p}_R(\alpha_0,\alpha_1)\hookrightarrow C^{0,\gamma}([0,R]),
\end{equation*}
Moreover, for all $u\in X^{1,p}_R$, we have
\begin{equation*}
|u(r)-u(s)|\leq C|r-s|^\gamma\|u'\|_{L^p_{\alpha_1}}\quad\forall r,s\in [0,R],
\end{equation*}
where $\gamma=\min\{1-(\alpha_1+1)/p,1-1/p\}$ and $C$ is a positive constant depending only on $p$ and $\min\{\alpha_1,0\}$.
\end{theo}
\begin{proof}
Since $X^{1,p}_R(\alpha_0,\alpha_1)\hookrightarrow X^{1,p}_R(\alpha_0,0)$ for $\alpha_1<0$, we can suppose $\alpha_1\geq0$. By Lemma \ref{lemmamorrey} we have
\begin{equation}\label{eqqmorreu}
|u(r)-u(s)|\leq C\|u'\|_{L^p_{\alpha_1}}|r-s|^{\gamma},\quad\forall r,s\in[0,R].
\end{equation}
Thus,
\begin{equation*}
|u(r)|\leq |u(R)|+C\|u'\|_{L^p_{\alpha_1}}R^{\gamma},\quad\forall r\in[0,R].
\end{equation*}
Using Lemma \ref{lemma31} we have the continuous embedding
\begin{equation*}
X^{1,p}_R(\alpha_0,\alpha_1)\hookrightarrow L^\infty([0,R]),
\end{equation*}
which together with \eqref{eqqmorreu} implies the embedding
\begin{equation*}
X^{1,p}_R(\alpha_0,\alpha_1)\hookrightarrow C^{0,\gamma}([0,R]),
\end{equation*}
and the proof is complete. \end{proof}

\begin{proof}[Proof of Theorem \ref{theo32}]
$a$) It is enough to show the continuous embedding for $q=p^*$. From Proposition \ref{prop32}, there exists $C>0$ such that
\begin{equation*}
|u(r)|^{\frac{(\theta-\alpha_k+kp)p}{\alpha_k-kp+1}}r^{\theta-\alpha_k+kp}\leq C\|u\|_{X_R^{k,p}}^{\frac{(\theta-\alpha_k+kp)p}{\alpha_k-kp+1}},\quad\forall r\in(0,R].
\end{equation*}
Therefore
\begin{align*}
\int_0^R|u(r)|^{\frac{(\theta+1)p}{\alpha_k-kp+1}}r^\theta\mathrm dr&\leq C\|u\|_{X_R^{k,p}}^{\frac{(\theta-\alpha_k+kp)p}{\alpha_k-kp+1}}\int_0^R|u(r)|^pr^{\alpha_k-kp}\mathrm dr\\
&\leq C\|u\|^{\frac{(\theta-\alpha_k+kp)p}{\alpha_k-kp+1}}_{X^{k,p}_R}\|u\|_{X^{k,p}_R}^p\\
&=C\|u\|_{X_R^{k,p}}^{\frac{(\theta+1)p}{\alpha_k-kp+1}},
\end{align*}
where the second inequality is a consequence of Proposition \ref{prop35}.

\vspace{0.2cm}

\noindent $b$) Using the Propositions \ref{prop33} and \ref{prop34}, the proof is straightforward. 

\vspace{0.2cm}

\noindent $c$) Let $u\in X^{k,p}_R(\alpha_0,\ldots,\alpha_k)$ with $\alpha_k-kp+1<0$. Denote $\ell=\lfloor\frac{\alpha_k+1}p\rfloor$. We divided the proof in two cases:

\medskip 

\noindent {\bf Case $\frac{\alpha_k+1}{p}\notin\mathbb Z$:} In this case $\alpha_k-\ell p+1>0$. By Theorem \ref{theo32},
\begin{equation*}
u^{(k-\ell)}\in X^{\ell,p}_R(\alpha_{k-\ell},\ldots,\alpha_k)\hookrightarrow L^p_{\alpha_k-\ell p}.
\end{equation*}
Thus, $u^{(k-\ell-1)}\in X^{1,p}_R(\alpha_{k-\ell-1},\alpha_k-\ell p)$. Then Theorem \ref{theomorrey} implies
\begin{equation*}
u^{(k-\ell-1)}\in X^{1,p}_R(\alpha_{k-\ell-1},\alpha_k-\ell p)\hookrightarrow C^{0,\gamma}([0,R]),
\end{equation*}
where $\gamma=\min\{1-\frac{\alpha_k-\ell p+1}{p},1-\frac{1}{p}\}$.

\medskip

\noindent {\bf Case $\frac{\alpha_k+1}{p}\in\mathbb Z$:} In this case $\alpha_k-\ell p+1=0$. By Theorem \ref{theo32}
\begin{equation*}
u^{(k-\ell)}\in X^{1,p}_R(\alpha_{k-\ell},\ldots,\alpha_k)\hookrightarrow L^q_0\hookrightarrow L^q_{p-1},\quad\forall q\in[1,\infty)
\end{equation*}
and
\begin{equation*}
u^{(k-\ell-1)}\in X^{1,p}_R(\alpha_{k-\ell-1},p-1)\hookrightarrow L^q_0,\quad\forall q\in[1,\infty).
\end{equation*}
Then Theorem \ref{theomorrey} concludes, for each $q\in(1,\infty)$,
\begin{equation*}
u^{(k-\ell-1)}\in X^{1,q}_R(0,0)\hookrightarrow C^{0,1-\frac1{q}}([0,R]).
\end{equation*}
Taking $q\to\infty$ we obtain $u^{(k-\ell-1)}\in C^{0,\gamma}([0,R])$ for each $\gamma\in(0,1)$.
\end{proof}

\section{Adams-Trudinger-Moser inequality for higher derivatives}\label{MTH}

Throughout this Section, we prove Theorem \ref{theo2} and \ref{theo3}. We focus on answering the following question: considering the space $X^{k,p}_R$ in the Adams-Trudinger-Moser case ($\alpha_k-kp+1=0$), for which $\mu>0$ the supremum
\begin{equation*}
\sup_{\|u\|_{X_R^{k,p}}\leq1}\int_0^Re^{\mu|u|^{\frac{p}{p-1}}}r^\theta \mathrm dr
\end{equation*}
is finite? Firstly, we need to check that $\int_0^R\exp(\mu|u|^{\frac{p}{p-1}})r^\theta\mathrm dr$ is finite for every $u\in X^{k,p}_R$. That is the item $(a)$ of Theorem \ref{theo2} and Proposition \ref{prop50} concludes it. We use Proposition \ref{prop51} to prove Proposition \ref{propitema}, which is the item $(b)$ of Theorem \ref{theo2}. At last, we prove item $(c)$ of Theorem~\ref{theo2} in Proposition~\ref{prop54}.

\begin{prop}\label{prop50}
Let $X_R^{k,p}(\alpha_0,\ldots,\alpha_k)$ with $\alpha_k-kp+1=0$ and $\theta>-1$. Then for any $\mu\geq0$ we have $\exp(\mu|u|^{p'})\in L^1_\theta$ for all $u\in X^{k,p}_R$.
\end{prop}
\begin{proof}
Let $u\in X^{k,p}_R$ and $\theta>-1$. We claim that
\begin{equation}\label{eqq34}
\lim_{r\to0}\dfrac{u(r)}{|\log r|^{\frac{p-1}{p}}}=0.
\end{equation}
Indeed, given $\varepsilon>0$ there exists $T=T(\varepsilon)>0$ such that $\int_0^T|u^{(k)}(r)|^pr^{kp-1}\mathrm dr<\varepsilon$. Using \eqref{eq15}, \eqref{eq16} and \eqref{eq17} (with $R=T$) we have \eqref{eqq34}. Fix $\delta=(\frac{\theta+1}{2\mu})^{(p-1)/p}$. By \eqref{eqq34}, there exists $T>0$ such that $|u(t)|\leq\delta|\log t|^{\frac{p-1}p}$ for all $t\leq T$. Then
\begin{equation*}
\int_0^Te^{\mu|u|^{p'}}r^{\theta}\mathrm dr\leq \int_0^Tr^{\frac{\theta-1}{2}}\mathrm dr<\infty.
\end{equation*}
This concludes $\exp(\mu|u|^{p'})\in L^1_\theta$ since Proposition \ref{prop33} estimates (by a constant depending on $T$) the term $\int_T^Re^{\mu|u|^{p'}}\mathrm dr$.
\end{proof}

\begin{prop}\label{prop51}
Let $X_R^{k,p}(\alpha_0,\ldots,\alpha_k)$ with $\alpha_k-kp+1=0$ and $\theta>-1$. If $\mu<\mu_0$, 
then
\begin{equation*}
\sup_{\|u\|_{X_R^{k,p}}\leq1}\int_0^Re^{\mu|u|^{\frac{p}{p-1}}}r^{\theta}\mathrm dr<\infty.
\end{equation*}
\end{prop}
\begin{proof}
Based on the proof of \cite[Theorem 2.6]{MR2482538}, we will find $\mu_0$ such that for any $\mu<\mu_0$ we have \eqref{e33}. Firstly we consider the following inequality
\begin{equation}\label{provisorio}
(a+b)^q\leq (1+\varepsilon)a^q+c_q\left(1+\dfrac1{\varepsilon^{q-1}}\right)b^q,\quad\forall a,b\geq0 \mbox{ and }\varepsilon>0,
\end{equation}
where $c_q>0$ depends only on $q\geq1$.

Let $u\in X_R^{k,p}$ with $\|u\|_{X_R^{k,p}}\leq1$. By the proof of Proposition \ref{prop33},
\begin{equation*}
|u(t)|\leq\left( \frac1{(k-1)!}\left|\log\frac tR\right|^{\frac{p-1}p}+C\right)\|u\|_{X_R^{k,p}},\quad\forall t\in(0,R].
\end{equation*}
Using \eqref{provisorio} we conclude
\begin{equation*}
|u(t)|^\frac{p}{p-1}\leq \frac{1+\varepsilon}{[(k-1)!]^{\frac{p}{p-1}}}\left|\log\frac{t}R\right|+c_{\varepsilon,p}\quad\forall t\in(0,R].
\end{equation*}
Then,
\begin{align*}
\int_0^Re^{\mu|u|^{\frac{p}{p-1}}}r^{\theta}\mathrm dr&\leq e^{\mu c_{p,\varepsilon}}R^{\frac{1+\varepsilon}{[(k-1)!]^{p/(p-1)}}}\int_0^Rr^\theta r^{-\frac{1+\varepsilon}{[(k-1)!]^{p/(p-1)}}\mu}\mathrm dr\\
&= e^{\mu c_{p,\varepsilon}}R^{\theta+1}\int_0^1r^{\theta -\frac{1+\varepsilon}{[(k-1)!]^{p/(p-1)}}\mu}\mathrm dr.
\end{align*}
Note that the last integral is finite if and only if 
\begin{equation*}
\theta-\dfrac{1+\varepsilon}{[(k-1)!]^{p/(p-1)}}\mu>-1.
\end{equation*}
Since 
\begin{equation*}
    \mu<\mu_0=(\theta+1)[(k-1)!]^{p/(p-1)},
\end{equation*}
we get the desired result by taking $\varepsilon>0$ sufficiently small.
\end{proof}

\begin{prop}\label{propitema}
Let $X^{k,p}_R(\alpha_0,\ldots,\alpha_k)$ with $\alpha_k-kp+1=0$ and $\theta>-1$. For every $0\leq\mu<\mu_0$,
\begin{equation*}
\sup_{\|u\|_{X^{k,p}_R}\leq1}\int_0^Re^{\mu|u|^{p'}}r^\theta\mathrm dr
\end{equation*}
is attained. Moreover, if $u$ is a maximizer, then $u$ can be chosen nonnegative and $\|u\|_{X^{k,p}_R}=1$.
\end{prop}
\begin{proof}
Let $(u_n)$ in $X^{k,p}_R$ be a maximizing sequence with $\|u_n\|_{X^{k,p}_R}\leq1$. By Theorem \ref{theo32} there exists $u_0\in X^{k,p}_R$ such that, up to subsequence, $u_n\rightharpoonup u_0$ in $X^{k,p}_R$ and $u_n\rightarrow u_0$ in $L_\theta^q(0,R)$ for all $q\geq1$. From $|e^x-e^y|\leq|x-y|(e^x+e^y)$ we get
\begin{equation}
\left|\int_0^R\left(e^{\mu|u_n|^{p'}}-e^{\mu|u_0|^{p'}}\right)r^\theta\mathrm dr\right|\leq\mu\int_0^R\left||u_n|^{p'}-|u_0|^{p'}\right|\left(e^{\mu|u_n|^{p'}}+e^{\mu|u_0|^{p'}}\right)r^{\theta}\mathrm dr.\label{eqtx}
\end{equation}
Fixing $q>1$ with $q\mu<\mu_0$ we have

\begin{align}
    \int_0^R\left||u_n|^{p'}-|u_0|^{p'}\right|e^{\mu|u_n|^{p'}}r^\theta\mathrm dr&\leq \left(\int_0^R\left||u_n|^{p'}-|u_0|^{p'}\right|^{q'}r^\theta\mathrm dr\right)^{\frac1{q'}}\left(\int_0^Re^{q\mu|u_n|^{p'}}r^{\theta}\mathrm dr\right)^{\frac1q}\nonumber\\
&\leq\|u_n-u_0\|_{L^{p'q'}_\theta}^{p'}\left(\int_0^Re^{q\mu|u_n|^{p'}}r^\theta\mathrm dr\right)^{\frac1q}.\label{eqt10}
\end{align}
By Proposition \ref{prop51}, the integral on the right side is uniformly bounded in $n$. Using \eqref{eqtx} and \eqref{eqt10} with $u_n\to u_0$ in $L^{p'q'}_\theta$ we conclude
\begin{equation*}
\sup_{\|u\|_{X^{k,p}_R}\leq1}\int_0^Re^{\mu|u|^{p'}}r^\theta\mathrm dr=\lim\int_0^Re^{\mu|u_n|^{p'}}r^\theta\mathrm dr=\int_0^Re^{\mu|u_0|^{p'}}r^\theta\mathrm dr.
\end{equation*}

To prove $\|u_0\|_{X^{k,p}_R}=1$, suppose $\|u_0\|_{X^{k,p}_R}<1$ and take $v=u_0/\|u_0\|_{X^{k,p}_R}$ to get a contradiction with the fact that $u$ attains the supremum. We can assume $u_0$ nonnegative changing $u_0$ for $|u_0|$ if necessary.
\end{proof}

\begin{prop}\label{prop54} Let $X_R^{k,p}(\alpha_0,\ldots,\alpha_k)$ such that $\alpha_i-ip+1>0$ for all $i=0,\ldots,k-1$ and $\alpha_k-kp+1=0$. If $\theta>-1$, then
\begin{equation*}
\sup_{\|u\|_{X_R^{k,p}}\leq1}\int_0^Re^{\mu|u|^{\frac{p}{p-1}}}r^{\theta}\mathrm dr=\infty,\quad\forall \mu>\mu_0.
\end{equation*}

\end{prop}

\begin{proof}
Our proof is based in \cite{MR0960950}. Let $\phi\in C^\infty[0,1]$ such that
\begin{equation*}
\left\{\begin{array}{ll}
\phi(0)=\phi'(0)=\cdots=\phi^{(k+1)}(0)=0,\ \phi'\geq0,\\
\phi(1)=\phi'(1)=1,\ \phi''(1)=\cdots\phi^{(k-1)}(1)=0.
\end{array}\right.
\end{equation*}
Consider $0<\varepsilon<1/2$ and
\begin{equation*}
H(t)=\left\{\begin{array}{llll}
     \varepsilon\phi\left(\dfrac{t}{\varepsilon}\right),&\mbox{if }0<t\leq\varepsilon  \\
     t,&\mbox{if }\varepsilon<t\leq1-\varepsilon,\\
     1-\varepsilon\phi\left(\dfrac{1-t}{\varepsilon}\right),&\mbox{if }1-\varepsilon<t\leq1,\\
     1,&\mbox{if }t>1.
\end{array}\right.
\end{equation*}
Let $m\in\mathbb N$ and the sequence $(\psi_{m,\varepsilon})_m$ given by
\begin{equation*}
\psi_{m,\varepsilon}(r)=H\left(\dfrac{\log\frac{R}{r}}{\log m}\right),\quad r>0.
\end{equation*}
We claim that
\begin{equation}\label{eq43}
\|\psi_{m,\varepsilon}\|_{X^{k,p}_R}^p\leq[(k-1)!]^p(\log m)^{1-p}\left[1+2^{p}\varepsilon\|\phi'\|^p_\infty+O\left((\log m)^{-1}\right)\right].
\end{equation}
To deduce \eqref{eq43}, we first check that
\begin{equation}\label{eq44}
\|\psi_{m,\varepsilon}\|_{L^p_{\alpha_0}}^p=O\left((\log m)^{-p}\right).
\end{equation}
Indeed, using the following change of variables $t=\log\frac{R}{r}/\log m$ we have
\begin{align}
\|\psi_{m,\varepsilon}&\|_{L^p_{\alpha_0}}^p=\int_0^R\left|H\left(\dfrac{\log\frac{R}{r}}{\log m}\right)\right|^pr^{\alpha_0}\mathrm dr\nonumber\\
&=R^{\alpha_0+1}\log m\Bigg[\int_0^\varepsilon\left|\varepsilon\phi\left(\dfrac{t}\varepsilon\right)\right|^pm^{-(\alpha_0+1)t}\mathrm dt+\int_\varepsilon^{1-\varepsilon}t^pm^{-(\alpha_0+1)t}\mathrm dt\nonumber\\
&\quad+\int_{1-\varepsilon}^1\left|1-\varepsilon\phi\left(\dfrac{1-t}{\varepsilon}\right)\right|^pm^{-(\alpha_0+1)t}\mathrm dt+\int_1^\infty m^{-(\alpha_0+1)t}\mathrm dt\Bigg]\nonumber\\
&\leq R^{\alpha_0+1}\log m\left[\int_0^\varepsilon\left|\varepsilon\phi\left(\dfrac{t}\varepsilon\right)\right|^pm^{-(\alpha_0+1)t}\mathrm dt+\int_\varepsilon^\infty m^{-(\alpha_0+1)t}\mathrm dt\right]\nonumber\\
&=R^{\alpha_0+1}\log m\int_0^\varepsilon\left|\varepsilon\phi\left(\dfrac{t}\varepsilon\right)\right|^pm^{-(\alpha_0+1)t}\mathrm dt+O\left((\log m)^{-p}\right).\label{eq451}
\end{align}
Fix $\gamma\in (\frac{1}{k+1},1)$. Since $\phi\left((\log m)^{-\gamma}\right)\leq (\log m)^{-\gamma (k+1)}$ for large $m$, we get
\begin{align}
\int_0^\varepsilon\left|\varepsilon\phi\left(\dfrac{t}\varepsilon\right)\right|^pm^{-(\alpha_0+1)t}\mathrm dt&\leq\varepsilon^p\int_0^{\varepsilon(\log m)^{-\gamma}}\left|\phi\left((\log m)^{-\gamma}\right)\right|^pm^{-(\alpha_0+1)t}\mathrm dt\nonumber\\
&\qquad+\varepsilon^p\int_{\varepsilon(\log m)^{-\gamma}}^\varepsilon m^{-(\alpha_0+1)t}\mathrm dt\nonumber\\
&\leq \varepsilon^p(\log m)^{-\gamma p(k+1)}\dfrac1{(\alpha_0+1)\log m}\nonumber\\
&\qquad+\varepsilon^p\dfrac{1}{(\alpha_0+1)e^{\varepsilon(\alpha_0+1)(\log m)^{1-\gamma}}\log m}\nonumber\\
&=O\left((\log m)^{-p-1}\right)\label{eq46}
\end{align}
for large $m$. \eqref{eq451} with \eqref{eq46} concludes \eqref{eq44}.

The task is now to estimate $\|\psi_{m,\varepsilon}^{(i)}\|_{L^p_{\alpha_i}}$ for each $i=1,\ldots,k$. By induction on $i=1,\ldots,k$ we have
\begin{equation*}
\psi_{m,\varepsilon}^{(i)}(r)=\dfrac1{r^i}\sum_{j=1}^i\dfrac{c(j,i)}{(\log m)^j}H^{(j)}\left(\dfrac{\log\frac{R}{r}}{\log m}\right),
\end{equation*}
where $c(j,i)$ satisfies $c(1,1)=-1$, $c(1,i)=(-1)^i(i-1)!$, $c(i,i)=(-1)^i$ and $c(j,i+1)=-ic(j,i)-c(j-1,i)$ for each $j=2,\ldots,i$. Then
\begin{align}
\|\psi_{m,\varepsilon}^{(i)}\|^p_{L^p_{\alpha_i}}&\leq \left(\dfrac{(i-1)!}{\log m}\right)^p\int_0^R\left|H'\left(\dfrac{\log\frac{R}{r}}{\log m}\right)+O\left((\log m)^{-1}\right)\right|^pr^{\alpha_i-ip}\mathrm dr\nonumber\\
&\leq\left(\dfrac{(i-1)!}{\log m}\right)^p\Bigg[\left|\|\phi'\|_\infty+O\left((\log m)^{-1}\right)\right|^p\int_{\frac{R}{m}}^{Rm^{\varepsilon-1}}r^{\alpha_i-ip}\mathrm dr\nonumber\\
&\quad+\int_{Rm^{\varepsilon-1}}^{Rm^{-\varepsilon}}r^{\alpha_i-ip}\mathrm dr+\left|\|\phi'\|_\infty+O\left((\log m)^{-1}\right)\right|^p\int_{Rm^{-\varepsilon}}^Rr^{\alpha_i-ip+1}\mathrm dr\Bigg].\label{eq47}
\end{align}
For $i=1,\ldots,k-1$, \eqref{eq47} implies
\begin{align}
\|\psi^{(i)}_{m,\varepsilon}\|^p_{L^p_{\alpha_1}}&\leq\left(\dfrac{(i-1)!}{\log m}\right)^p\Bigg[\dfrac{R^{\alpha_i-ip+1}\left|\|\phi'\|_\infty+O\left((\log m)^{-1}\right)\right|^p}{(\alpha_i-ip+1)m^{(1-\varepsilon)(\alpha_i-ip+1)}}\nonumber\\
&\quad+\dfrac{R^{\alpha_i-ip+1}}{m^{\varepsilon(\alpha_i-ip+1}}+\dfrac{R^{\alpha_i-ip+1}\left|\|\phi'\|_\infty+O\left((\log m)^{-1}\right)\right|^p}{\alpha_i-ip+1}\Bigg]\nonumber\\
&=O\left((\log m)^{-p}\right).\label{eq48}
\end{align}
Using \eqref{eq47} we get
\begin{align}
\|\psi^{(k)}_{m,\varepsilon}\|^p_{L^p_{\alpha_k}}&\leq \left(\dfrac{(k-1)!}{\log m}\right)^p\Bigg[\varepsilon\log m\left|\|\phi'\|_\infty+O\left((\log m)^{-1}\right)\right|^p+\log m\nonumber\\
&\quad+\varepsilon\log m\left|\|\phi'\|_\infty+O\left((\log m)^{-1}\right)\right|^p\Bigg]\nonumber\\
&\leq \left[(k-1)!\right]^p(\log m)^{1-p}\left[1+2^p\varepsilon\|\phi'\|_\infty^p+O\left((\log m)^{-p}\right)\right].\label{eq49}
\end{align}
Finally, \eqref{eq43} follows by \eqref{eq44}, \eqref{eq48} and \eqref{eq49}.

Set
\begin{equation*}
u_{m,\varepsilon}(r)=\dfrac{\psi_{m,\varepsilon}(r)}{\|\psi_{m,\varepsilon}\|_{X^{k,p}_R}}.
\end{equation*}
Given $\mu>\mu_0=(\theta+1)[(k-1)!]^{p'}$, \eqref{eq43} guarantees
\begin{align*}
\int_0^Re^{\mu u_{m,\varepsilon}^{p'}}r^\theta\mathrm dr&\geq\int_0^{\frac{R}m}e^{\mu\|\psi_{m,\varepsilon}\|_{X^{k,p}_R}^{-p'}}r^\theta\mathrm dr\\
&\geq e^{\frac{\mu\log m}{[(k-1)!]^{p'}[1+2^p\varepsilon\|\phi'\|_\infty^p+O\left((\log m)^{-1}\right)}-(\theta+1)\log m}\dfrac{R^{\theta+1}}{\theta+1}\\
&=\dfrac{R^{\theta+1}}{\theta+1}e^{(\theta+1)\log m\left(\frac{\mu}{\mu_0}\frac{1}{1+2^p\varepsilon\|\phi'\|^p_\infty+O\left((\log m)^{-1}\right)}-1\right)}.
\end{align*}
Therefore, taking $\varepsilon>0$ sufficiently small we have that the right term tends to be infinite when $m\to\infty$.

\end{proof}

Before we prove the Theorem \ref{theo3}, we need to consider Theorem 1.2 in \cite{MR3209335}:

\begin{theo}\label{theo12jmbo}
Let $p\geq2$ be a real number. Then there exists a constant $c_1$ depending on $p$ such that
\begin{equation*}
\sup_{w\in\mathcal K}\int_0^\infty e^{\rho w^{\frac{p}{p-1}}(t)-t}\mathrm dt\left\{\begin{array}{ll}
     \leq c_1&\mbox{if }\rho\leq1,  \\
     =\infty&\mbox{if }\rho>1, 
\end{array}\right.
\end{equation*}
where $\mathcal K:=\{w\in C^1([0,\infty))\mbox{ with }w(0)=0,\ w\geq0,\ \int_0^\infty|w'|^p\mathrm dt\leq1\}$.
\end{theo}

Our idea consists in apply the same change of variables as Moser \cite{MR0301504} and extends our function $w$ to $\widetilde w$ such that satisfies the boundary condition. Therefore using the Theorem with boundary condition (Theorem \ref{theo12jmbo}) we deduce our Theorem \ref{theo3}. The condition $u(R)\leq Au(r)$ is necessary to guarantee that $\int_0^\infty|\widetilde w'(t)|^p\mathrm dt\leq1$.

\begin{proof}[Proof of Theorem \ref{theo3}]
Changing $u$ by $|u|$, we can suppose $u\geq0$. Define $w(t)=(\theta+1)^{\frac{p-1}{p}}u(r)$ where $t=-(\theta+1)\log(r/R)$. Then
\begin{align}
\int_0^Re^{(\theta+1)|u(r)|^{p'}}r^\theta\mathrm dr&=\dfrac{R^{\theta+1}}{\theta+1}\int_0^Re^{\left|(\theta+1)^{\frac{p-1}p}u(r)\right|^{p'}}\left(\dfrac{r}{R}\right)^{\theta+1}\dfrac{(\theta+1)\mathrm dr}r\nonumber\\
&=\dfrac{R^{\theta+1}}{\theta+1}\int_0^\infty e^{|w(t)|^{p'}-t}\mathrm dt.\label{ew253}
\end{align}
On the other hand,
\begin{align*}
\|u\|^p_{X^{1,p}_R}&=\int_0^R|u(r)|^pr^{\alpha_0+1}\dfrac{\mathrm dr}r+\int_0^R|u'(r)|^pr^{p}\dfrac{\mathrm dr}r\\
&=\dfrac{R^{\alpha_0+1}}{(\theta+1)^p}\int_0^\infty|w(t)|^pe^{-\frac{\alpha_0+1}{\theta+1}t}\mathrm dt+\int_0^\infty|w'(t)|^p\mathrm dt.
\end{align*}
Thus,
\begin{equation*}
\sup_{u\in\mathcal K_A}\int_0^Re^{(\theta+1)|u|^{p'}}r^\theta\mathrm dr=\dfrac{R^{\theta+1}}{\theta+1}\sup_{w\in\widetilde{\mathcal K}_A}\int_0^\infty e^{|w(t)|^{p'}-t}\mathrm dt,
\end{equation*}
where
\begin{align*}
\widetilde{\mathcal K}_A:=\{w\in &C((0,\infty);\mathbb R)\colon w\geq0,\ w'\mbox{ exists a.e.},\\
&\ \|w\|\leq1\mbox{ and }w(0)\leq Aw(t)\ \forall t\in[0,\infty)\},
\end{align*}
and
\begin{equation*}
\|w\|:=\left(\dfrac{R^{\alpha_0+1}}{(\theta+1)^p}\int_0^\infty|w(t)|^pe^{-\frac{\alpha_0+1}{\theta+1}t}\mathrm dt+\int_0^\infty|w'(t)|^p\mathrm dt\right)^{1/p}.
\end{equation*}

Let $w\in\widetilde{\mathcal K}_A$ with $\|w\|\leq1$. Define 
\begin{equation*}
C_1=(\theta+1)\left[\dfrac{(\alpha_0+1)A^p}{R^{\alpha_0+1}}\right]^{\frac1{p-1}}
\end{equation*}
and $\widetilde w\colon [0,\infty)\to\mathbb R$ given by
\begin{equation*}
\widetilde w(t)=\left\{\begin{array}{cc}
     \dfrac{w(0)t}{C_1}&\mbox{if }0\leq t\leq C_1  \\
     w(t-C_1)&\mbox{if }t\geq C_1.
\end{array}\right.
\end{equation*}
Since
\begin{align*}
\int_0^\infty|\widetilde w'(t)|^p\mathrm dt&=\dfrac{|w(0)|^p}{C_1^{p-1}}+\int_0^\infty|w'(t)|^p\mathrm dt\\
&=\dfrac{\alpha_0+1}{(\theta+1)C_1^{p-1}}\int_0^\infty|w(0)|^pe^{-\frac{\alpha_0+1}{\theta+1}t}\mathrm dt+\int_0^\infty|w'(t)|^p\mathrm dt\\
&\leq \dfrac{(\alpha_0+1)A^p}{(\theta+1)C_1^{p-1}}\int_0^\infty|w(t)|^pe^{-\frac{\alpha_0+1}{\theta+1}t}\mathrm dt+\int_0^\infty|w'(t)|^p\mathrm dt\\
&\leq 1,
\end{align*}
Theorem \ref{theo12jmbo} implies
\begin{equation*}
\int_0^\infty e^{|\widetilde w(t)|^{p'}-t}\mathrm dt\leq C_2,
\end{equation*}
for some constant $C_2>0$ depending only on $p$. Therefore,
\begin{equation*}
C_2\geq \int_{C_1}^\infty e^{|w(t-C_1)|^{p'}-t}\mathrm dt=e^{-C_1}\int_0^\infty e^{|w(t)|^{p'}-t}\mathrm dt.
\end{equation*}
\end{proof}

Here we show that the exponential function $\Psi(t)$ given by \eqref{defphi} is optimal in the sense that it is the maximal growth for the embedding of weighted Sobolev spaces $X^{k,p}_{0,R}$ into weighted Orlicz spaces. More precisely, we show the Theorem \ref{theo51}.
\begin{lemma}\label{lemma55}
Suppose $\alpha_k-kp+1=0$ and $\alpha_i-ip+1>0$ for all $i=0,\ldots,k-1$. Then there exists a bounded sequence $(u_n)\subset X^{k,p}_{0,R}$ satisfying
\begin{equation*}
    \int_0^Re^{|u_n|^{p'}}r^\theta\mathrm dr\overset{n\to\infty}\longrightarrow\infty.
\end{equation*}
\end{lemma}
\begin{proof}
Fixed $\mu>\mu_0$, the proof of Proposition \ref{prop54} guarantees a sequence $(v_n)$ in $X^{k,p}_{0,R}$ such that $\|v_n\|_{X^{k,p}_R}=1$ and
\begin{equation*}
    \int_0^Re^{\mu|v_n|^{\frac{p}{p-1}}}r^\theta\mathrm dr\overset{n\to\infty}\longrightarrow\infty.
\end{equation*}
Defining $u_n:=\mu^{\frac{p-1}p}v_n$, we obtain that $(u_n)$ is bounded in $X^{k,p}_R$ with
\begin{equation*}
    \int_0^Re^{|u_n|^{\frac{p}{p-1}}}r^\theta\mathrm dr\overset{n\to\infty}\longrightarrow\infty.
\end{equation*}
\end{proof}
\begin{lemma}\label{lemma551}
    Suppose two $N$-functions such that $\Phi\prec\Psi$. If the space $X^{k,p}_{0,R}(\alpha_0,\ldots,\alpha_k)$ is continuously embedded in $L_\Psi(\theta)$, then for any $\nu>0$ we have
    \begin{equation*}
        \sup_{\|u\|_{X^{k,p}_{0,R}}\leq\nu}\int_0^R\Phi(u)r^\theta\mathrm dr<\infty.
    \end{equation*}
\end{lemma}
\begin{proof}
    There exists a constant $C_0>0$ such that
    \begin{equation}\label{eqbabubabu}
    \|u\|_{L_\Psi(\theta)}\leq C_0\|u\|_{X^{k,p}_{0,R}},\quad\forall u\in X^{k,p}_{0,R}.
    \end{equation}
    From $\Phi\prec\Psi$ we obtain $T_\nu$ satisfying
    \begin{equation*}
        \Psi(t)\leq\Psi\left(\dfrac{t}{C_0\nu}\right),\quad\forall t\geq T_\nu.
    \end{equation*}
    Let $u\in X^{k,p}_{0,R}$ such that $\|u\|_{X^{k,p}_{0,R}}\leq\nu$. Using \eqref{eqbabubabu} we have
    \begin{equation*}
    \int_0^R\Psi\left(\dfrac{u}{C_0\nu}\right)r^\theta\mathrm dr\leq1.
    \end{equation*}
    Therefore,
    \begin{align*}
    \int_0^R\Phi(u)r^\theta\mathrm dr&=\int_{|u|\geq T_\nu}\Phi(u)r^\theta\mathrm dr+\int_{|u|< T_\nu}\Phi(u)r^\theta\mathrm dr\\
    &\leq 1+\dfrac{\Phi(T_\nu)R^{\theta+1}}{\theta+1}.
    \end{align*}
\end{proof}
  \begin{proof}[Proof of the Theorem \ref{theo5}]
      The proof is straightforward from Lemmas \ref{lemma55} and \ref{lemma551}.
  \end{proof}

\section{Adams-type inequality for Weighted Sobolev spaces with Navier Boundary condition}\label{AIN}

\subsection{Equivalence between usual norm and k-gradient norm}

In this subsection, our task is to prove the following Proposition.

\begin{prop}\label{propequivnormkgrad}
Suppose $\alpha_k-(k-1)p+1>0$, $\gamma>(\alpha_k-p+1)/p$ and $\alpha_i\geq\alpha_k-(k-i)p$ for all $i=0,\ldots,k$. $\|\nabla^k_{\gamma}\cdot\|_{L^p_{\alpha_k}}$ is a norm equivalent to $\|\cdot\|_{X^{k,p}_R}$ in $X^{k,p}_{\mathcal N,\gamma,R}(\alpha_0,\ldots,\alpha_k)$.
\end{prop}

The next Proposition is the case $k=2$ of Proposition \ref{propequivnormkgrad} but with a generalized radial elliptic operator $L$ instead of $\Delta_\gamma$.

\begin{prop}\label{propequivnormL}
Consider the weighted Sobolev space $X^{2,p}_R(\alpha_0,\alpha_1,\alpha_2)\cap X^{1,p}_{0,R}(\alpha_0,\alpha_1)$ with $1\leq p<\infty$, $\alpha_0\geq\alpha_2-2p$ and $\alpha_1\geq\alpha_2-p$. Let $Lu=-r^{-\gamma}(r^\alpha u')'$ be an operator with $\alpha,\gamma>0$ and $\alpha>(\alpha_2-p+1)/p$. Suppose $\alpha_2-p+1\geq0$. Set $\eta=\alpha_2+p(\gamma-\alpha)$. Then the norm
\begin{equation*}
\|u\|_L:=\left(\int_0^R|Lu|^pr^{\eta}\mathrm dr\right)^{1/p}
\end{equation*}
is equivalent to the usual norm $\|u\|_{X^{2,p}_R}$ in $X^{2,p}_R(\alpha_0,\alpha_1,\alpha_2)\cap X^{1,p}_{0,R}(\alpha_0,\alpha_1)$.
\end{prop}

The following Lemma is equivalent to $L^p$ regularity in classic Sobolev spaces. Under some conditions, we prove that given $Lu=v\in L^p_\eta$ with $u(R)=u'(0)=0$ we have $u\in X^{2,p}_R\cap X^{1,p}_{0,R}$.

\begin{lemma}\label{lemmafjs}
Let $\alpha,\gamma,\alpha_2\in\mathbb R$, $v\in L^p_\eta(0,R)$ with $1\leq p<\infty$ and $\eta=\alpha_2+p(\gamma-\alpha)$. Define
\begin{equation}\label{expressionu}
u(r)=\int_r^Rt^{-\alpha}\int_0^tv(s)s^\gamma\mathrm ds\mathrm dt\quad\mathrm{a.e.}\ r\in(0,R).
\end{equation}
If $\alpha>(\alpha_2-p+1)/p$, then $u\in X^{2,p}_R(\alpha_2-2p,\alpha_2-p,\alpha_2)$ with $\|u\|_{X^{2,p}_R}\leq C\|v\|_{L^p_\eta}$, where $C=C(\alpha,\gamma,\alpha_2,p,R)>0$.
\end{lemma}
\begin{proof}
We can weakly derivate $u$ twice and those derivatives are given by
\begin{equation}\label{eu1}
u'(r)=-r^{-\alpha}\int_0^rv(s)s^\gamma\mathrm ds
\end{equation}
and
\begin{equation}\label{eu2}
u''(r)=\alpha r^{-(\alpha+1)}\int_0^rv(s)s^\gamma\mathrm ds-v(r)r^{\gamma-\alpha}.
\end{equation}
Note that $u\in AC^1_{\mathrm{loc}}((0,R])$. Since $u(R)=0$, Proposition \ref{prop21JMBO} guarantees $\|u\|_{L^p_{\alpha_2-2p}}\leq C\|u'\|_{L^p_{\alpha_2}}$. Then we are left with the task to prove that $u'\in L^p_{\alpha_2-p}$ and $u''\in L^p_{\alpha_2}$.

Firstly, let's prove that $u''\in L^p_{\alpha_2}$. Define $w(r)=\int_0^rv(s)s^\gamma\mathrm ds$. Since $w(0)=0$ and $\alpha>(\alpha_2-p+1)/p$, Proposition \ref{prop21JMBO} implies
\begin{align*}
\int_0^R\left|r^{-(\alpha+1)}\int_0^rv(s)s^\gamma\mathrm ds\right|^pr^{\alpha_2}\mathrm dr&=\int_0^R|w(r)|^pr^{\alpha_2-p(\alpha+1)}\mathrm dr\\
&\leq C\int_0^R|w'(r)|^pr^{\alpha_2-p\alpha}\mathrm dr=C\|v\|^p_{L^p_\eta}.
\end{align*}
Therefore, by \eqref{eu2},
\begin{align*}
    \int_0^R|u''(r)|^pr^{\alpha_2}\mathrm dr&\leq C\int_0^R\left|r^{-(\alpha+1)}\int_0^rv(s)s^\gamma\mathrm ds\right|^pr^{\alpha_2}\mathrm dr\\
    &\qquad+C\int_0^R|v(s)|^pr^{\alpha_2+p(\gamma-\alpha)}\mathrm dr\leq C\|v\|^p_{L^p_{\eta}}.
\end{align*}

Now we only need to check that $u'\in L^p_{\alpha_1}$. Applying Proposition \ref{prop21JMBO} as before, we get (by \eqref{eu1})
\begin{align*}
\int_0^R|u'(r)|^ pr^{\alpha_2-p}\mathrm dr&=\int_0^r|w(r)|^pr^{\alpha_2-p(\alpha+1)}\mathrm dr\leq C\int_0^R|w'(r)|^pr^{\alpha_2-p\alpha}\mathrm dr\\
&=C\|v\|^p_{L^p_\eta}.
\end{align*}
\end{proof}

\begin{proof}[Proof of Proposition \ref{propequivnormL}]
First, we need to check that $\|\cdot\|_L$ is a norm. All norm properties are straightforward except $\|u\|_L=0\Rightarrow u=0$. Let $u$ with $\|u\|_L=0$. Then $(r^\alpha u')'=0$. Since $u'\in X^{1,p}_R(\alpha_1,\alpha_2)$ and $\alpha_2-p+1\geq0$, by Radial Lemma (Propositions \ref{prop32}, \ref{prop33} and \ref{prop34}) with $\alpha>(\alpha_2-p+1)/p$ we get $\lim_{r\to0}r^\alpha u'(r)=0$. Thus $u'=0$. Using $u\in X^{1,p}_{0,R}$ we conclude $u=0$. Therefore, $\|\cdot\|_L$ is a norm.

Now let's prove that $\|u\|_L\leq C\|u\|_{X^{2,p}_R}$. Note that
\begin{equation}\label{eqisa1}
\|u\|^p_L=\int_0^R|\alpha u'+ru''|^pr^{\alpha_2-p}\mathrm dt\leq2^p\alpha^p\int_0^R|u'|^pr^{\alpha_2-p}\mathrm dr+2^p\int_0^R|u''|^pr^{\alpha_2}\mathrm dr.
\end{equation}
Since $u'\in X^{1,p}_R(\alpha_1,\alpha_2)$ and $\alpha_2-p+1\geq0$, Theorem \ref{theo32} guarantees $\|u'\|_{L^p_{\alpha_2-p}}\leq C\|u'\|_{X^{1,p}_R}$ for some $C>0$. Therefore, \eqref{eqisa1} implies $\|u\|_L\leq C\|u\|_{X^{2,p}_R}$.

It is sufficient to proof that $\|u\|_{X^{2,p}_R}\leq C\|u\|_L$. Our idea is to use the Open Mapping Theorem in the following operator
\begin{eqnarray*}
    L\colon X^{2,p}_R(\alpha_0,\alpha_1,\alpha_2)\cap X^{1,p}_{0,R}(\alpha_0,\alpha_1)&\longrightarrow&L^p_\eta\\
    u&\longmapsto&Lu.
\end{eqnarray*}
It is easy to see that $L$ is a linear map between Banach spaces. Throughout this demonstration, we have shown that $L$ is injective and continuous. For the surjective of $L$, for each $v\in L^p_\eta$, Lemma \ref{lemmafjs} gives $u$ such that $Lu=v$ and
\begin{equation*}
u\in X^{2,p}_R(\alpha_2-2p,\alpha_2-p,\alpha_2)\subset X^{2,p}_R(\alpha_0,\alpha_1,\alpha_2).
\end{equation*}
By \eqref{expressionu} we have $u\in X^{1,p}_{0,R}(\alpha_0,\alpha_1)$. Therefore the result follows by Open Mapping Theorem.
\end{proof}

To conclude Proposition \ref{propequivnormkgrad} we need three more Lemmas. The last one (Lemma \ref{qww}) is a generalization of Lemma \ref{lemmafjs} for all $k$ besides $L=\Delta_\gamma$.

\begin{lemma}\label{lemmajaosn}
Let $\gamma\in\mathbb R$, $j=0,\ldots,k$ and $X^{k,p}_R(\alpha_0,\ldots,\alpha_k)$ be the weighted Sobolev space such that $\alpha_i-(i-1)p+1\geq0$ for all $i=j,\ldots,k$. For each $u\in X^{k,p}_R$, we have $\nabla^j_\gamma u\in X^{k-j,p}_R(\alpha_j,\ldots,\alpha_k)$ with
\begin{equation*}
\|\nabla_\gamma^ju\|_{X^{k-j,p}_R}\leq C\|u\|_{X^{k,p}_R},
\end{equation*}
where $C>0$ does not depend on $u$.
\end{lemma}
\begin{proof}
Consider $\ell=0,\ldots,k-j$. Using induction on $j$ and $\ell$ we can prove that
\begin{equation}\label{doubleinduction}
\left(\nabla^j_\gamma u\right)^{(\ell)}=\sum_{i=0}^{j+\ell-1}C_{ijl}\dfrac{u^{(j+\ell-i)}}{r^i},
\end{equation}
for some $C_{ijl}=C_{ijl}(\gamma)\in\mathbb R$. It is enough to check that 
\begin{equation*}
\|(\nabla^j_\gamma u)^{(\ell)}\|_{L^p_{\alpha_{j+\ell}}}\leq C\|u\|_{X^{k,p}_R}.
\end{equation*}
Using \eqref{doubleinduction}, $\alpha_{j+\ell}-(j+\ell-1)p+1\geq0$ and Theorem \ref{theo32} (on $u^{(j+\ell-i)}\in X^{i,p}_R(\alpha_{j+\ell-i},\ldots,\alpha_{j+\ell})$) we have
\begin{equation*}
\|(\nabla^j_\gamma u)^{(\ell)}\|_{L^p_{\alpha_{j+\ell}}}^p\leq C\sum_{i=0}^{j+\ell-1}\int_0^R\left|\dfrac{u^{(j+\ell-i)}}{r^i}\right|^pr^{\alpha_{j+\ell}}\mathrm dr\leq C\sum_{i=0}^{j+\ell-1}\|u^{(j+\ell-i)}\|_{X^{i,p}_R}\leq C\|u\|_{X^{j+\ell,p}_R}.
\end{equation*}
\end{proof}
\begin{lemma}\label{corxnrbanach}
$X^{k,p}_{\mathcal N,\gamma,R}(\alpha_0,\ldots,\alpha_k)$ is a Banach space with the norm $\|\cdot\|_{X^{k,p}_R}$.
\end{lemma}
\begin{proof}
Let $(u_n)$ Cauchy sequence in $X^{k,p}_{\mathcal N,\gamma,R}$. Then $u_n\to u$ in $X^{k,p}_R$ for some $u\in X^{k,p}_R$. By \eqref{doubleinduction},
\begin{equation*}
|\Delta^j_{\gamma}u(R)|=|\nabla^{2j}_{\gamma}u_n(R)-\nabla^{2j}_{\gamma}u(R)|\leq \sum_{i=0}^{2j-1}C_{ij}\frac{|u_n^{(2j-i)}(R)-u^{(2j-i)}(R)|}{R^i},
\end{equation*}
for all $0\leq j\leq\lfloor\frac{k-1}{2}\rfloor$ and $n\in\mathbb N$. Using Lemma \ref{lemma31} we conclude 
\begin{equation*}
|\Delta^j_{\gamma} u(R)|\leq \sum_{i=0}^{2j-1}\dfrac{C_{ij}\widetilde C_{ij}\|u_n^{(2j-i)}-u^{(2j-i)}\|_{X^{1,p}_R}}{R^i}\overset{n\to\infty}\longrightarrow 0.
\end{equation*}
\end{proof}
\begin{lemma}\label{qww}
Let $v\in X^{k,p}_R(\alpha_k-kp,\ldots,\alpha_k)$ and
\begin{equation*}
u(r):=\int_r^Rt^{-\gamma}\int_0^tv(s)s^{\gamma}\mathrm ds\mathrm dt,\quad r\in(0,R),
\end{equation*}
where $\gamma>(\alpha_k-kp-p+1)/p$. Then \begin{equation*}u\in X^{k+2,p}_R(\alpha_k-(k+2)p,\alpha_k-(k+1)p,\alpha_k-kp,\ldots,\alpha_k).\end{equation*}
\end{lemma}
\begin{proof}
We claim that
\begin{equation}\label{qw1}
\dfrac{v(r)}{r^{k-i}}\in X^{i,p}_R(\alpha_k-ip,\ldots,\alpha_k),\quad\forall i=0,\ldots,k.
\end{equation}
Indeed, the case $i=0$ in \eqref{qw1} follows by $v\in L^p_{\alpha_k-kp}$. Suppose \eqref{qw1} holds for any $j=0,\ldots,i$. Our task is to prove that
\begin{equation*}
\dfrac{v(r)}{r^{k-i-1}}\in X^{i+1,p}_R(\alpha_k-(i+1)p,\ldots,\alpha_k).
\end{equation*}
Since $v(r)/r^{k-i-1}\in L^p_{\alpha_k-(i+1)p}$, we only need to check that
\begin{equation*}
\dfrac{v'(r)}{r^{k-i-1}}-(k-i-1)\dfrac{v(r)}{r^{k-i}}=\left(\dfrac{v(r)}{r^{k-i-1}}\right)'\in X^{i,p}_R(\alpha_k-ip,\ldots,\alpha_k).
\end{equation*}
Using induction hypothesis on $v'\in X^{k-1,p}_R$ and $v\in X^{k,p}_R$ we conclude \eqref{qw1}.

Let us verify that
\begin{equation}\label{qw2}
r^{-(\gamma+i)}\int_0^rv(s)s^{\gamma}\mathrm ds\in L^p_{\alpha_k-(k-i+1)p},\quad\forall i\in\mathbb Z.
\end{equation}
Denote $w(r)=\int_0^rv(s)s^{\gamma}\mathrm ds$. By Proposition \ref{prop21JMBO} and $\alpha_k-(\gamma+k+1)p+1<0$,
\begin{align*}
\int_0^R\left|r^{-(\gamma+i)}\int_0^rv(s)s^{\gamma}\mathrm ds\right|^pr^{\alpha_k-(k-i+1)p}\mathrm dr&=\int_0^R|w(r)|^pr^{\alpha_k-(\gamma+k+1)p}\mathrm dr\\
&\leq C\int_0^R|w'(r)|^pr^{\alpha_k-(\gamma+k)p}\mathrm dr\\
&=C\int_0^R|v(r)|^pr^{\alpha_k-kp}\mathrm dr\\
&=C\|v\|^p_{L^p_{\alpha_k-kp}},\quad\forall i\in\mathbb Z.
\end{align*}

By Lemma \ref{lemmafjs} we already have $u\in X^{2,p}_R(\alpha_k-(k+2)p,\alpha_k-(k+1)p,\alpha_k-kp)$. Then it is enough to prove that
\begin{equation*}
u''(r)=\gamma r^{-(\gamma+1)}\int_0^rv(s)s^{\gamma}\mathrm ds-v(r)\in X^{k,p}_R(\alpha_k-kp,\ldots,\alpha_k).
\end{equation*}
The proof is completed by showing that
\begin{equation}\label{qw3}
r^{-(\gamma+1)}\int_0^rv(s)s^{\gamma}\mathrm ds\in X^{k,p}_R(\alpha_k-kp,\ldots,\alpha_k).
\end{equation}
Indeed, for each $i\in\mathbb N\cup\{0\}$, set
\begin{equation*}
w_i(r):=r^{-(\gamma+i+1)}\int_0^rv(s)s^{\gamma}\mathrm ds.
\end{equation*}
\eqref{qw2} guarantees $w_i\in L^p_{\alpha_k-(k-i)p}$. Note that, using \eqref{qw1},
\begin{align*}
w_0\in X^{k,p}_R&\Leftrightarrow-(\gamma+1)r^{-(\gamma+i+2)}\int_0^rv(s)s^{\gamma}\mathrm ds+\dfrac{v(r)}{r}=w_0'\in X^{k-1,p}_R\\
&\Leftrightarrow w_1\in X^{k-1,p}_R\\
&\Leftrightarrow w_1'\in X^{k-2,p}_R\\
&\ \ \vdots\\
&\Leftrightarrow w_k\in L^p_{\alpha_k}.
\end{align*}
This concludes the Lemma.
\end{proof}

\begin{proof}[Proof of the Proposition \ref{propequivnormkgrad}]
As did in the proof of Proposition \ref{propequivnormL}, our task is to prove that
\begin{eqnarray*}
    \phi\colon X^{k,p}_{\mathcal N,\gamma,R}&\longrightarrow &L^p_{\alpha_k}\\
    u&\longmapsto&\nabla^k_{\gamma}u
\end{eqnarray*}
is an isomorphism. Note that $\phi$ is linear and, by Lemma \ref{lemmajaosn}, continuous. By Open Mapping Theorem and Lemma \ref{corxnrbanach}, it is enough to prove that $\phi$ is bijective. Let $u\in X^{k,p}_{\mathcal N,\gamma,R}$ with $\phi(u)=\nabla^k_{\gamma}u=0$. We claim that
\begin{equation}\label{eqclaimhsu}
r^{\gamma}\nabla^i_{\gamma}u\overset{r\to0}\longrightarrow0\quad\forall1\leq i\leq k-1.
\end{equation}
Indeed, by \eqref{doubleinduction},
\begin{equation*}
\left|r^{\gamma}\nabla^i_{\gamma}u\right|\leq\sum_{\ell=0}^{i-1}C_{i\ell}r^{\gamma-\ell}|u^{(i-\ell)}|.
\end{equation*}
For now suppose $\alpha_k-(k-1)p+1>0$. Since $u^{(i-\ell)}\in X^{k-i+\ell,p}_R$, Proposition \ref{prop32} implies $|u^{(i-\ell)}(r)|\leq C\|u\|_{X^{k,p}_R}r^{-\frac{\alpha_k-(k-i+\ell)p+1}{p}}$. Then
\begin{equation*}
\left|r^{\gamma}\nabla^i_{\gamma}u\right|\leq C\|u\|_{X^{k,p}_R}\sum_{\ell=0}^{i-1}r^{\gamma-\ell-\frac{\alpha_k-(k-i+\ell)p+1}p}\leq C\|u\|_{X^{k,p}_R}r^{\gamma-\frac{\alpha_k-(k-i)p+1}p}.
\end{equation*}
Thus \eqref{eqclaimhsu} follows by $\gamma>(\alpha_k-p+1)/p\geq(\alpha_k-(k-i)p+1)/p$. For the case $\alpha_k-(k-1)p+1=0$, following the same argument using Proposition \ref{prop33} and \ref{prop34} instead of Proposition \ref{prop32} we conclude \eqref{eqclaimhsu}.

Let $j$ be the integer such that $k=2j$ or $k=2j+1$. By $\nabla^k_{\gamma} u=0$ and $u\in X^{k,p}_{\mathcal N,\gamma,R}$ we have $\Delta^j_{\gamma}u=0$. $u\in X^{k,p}_{\mathcal N,\gamma,R}$ and \eqref{eqclaimhsu} guarantee that we can apply the following result
\begin{equation*}
\Delta_{\gamma}v=0,\ r^{\gamma}v'\overset{r\to0}\longrightarrow0\mbox{ and }v(R)=0\Rightarrow v=0
\end{equation*}
$j$-times on $\Delta^j_{\gamma}u=0$ to obtain $u=0$. This concludes that $\phi$ is injective.

Now let us prove that $\phi$ is surjective. Suppose $k=2j$. Given $v\in L^p_{\alpha_k}$, Lemma \ref{lemmafjs} implies that there exists $u_1\in X^{2,p}_R(\alpha_k-2p,\alpha_k-p,\alpha_k)\cap X^{1,p}_{0,R}$ such that $\Delta_{\gamma}u_1=v$. Again using Lemma \ref{lemmafjs} we get $u_2\in X^{2,p}_R(\alpha_k-4p,\alpha_k-3p,\alpha_k-2p)\cap X^{1,p}_{0,R}$ with $\Delta_{\gamma}u_2=u_1$. By Lemma \ref{qww} and $u_1\in X^{1,p}_{0,R}$ we have $u_2\in X^{4,p}_{\mathcal N,\gamma,R}(\alpha_k-4p,\ldots,\alpha_k)$ and $\Delta_{\gamma}^2u_2=v$. Proceeding with this argument, we obtain $u_j\in X^{2j,p}_{\mathcal N,\gamma,R}(\alpha_k-2jp,\ldots,\alpha_k)$ such that $\phi(u_j)=\Delta^j_{\gamma}u_j=v$.

Now suppose $k=2j+1$. Given $v\in L^p_{\alpha_k}$, set
\begin{equation*}
\widetilde u(r)=-\int_r^Rv(s)\mathrm ds.
\end{equation*}
By Proposition \ref{prop21JMBO}, $\|\widetilde u\|_{L^p_{\alpha_k-p}}\leq C\|v\|_{L^p_{\alpha_k}}$. Then $\widetilde u\in X^{1,p}_{0,R}(\alpha_k-p,\alpha_k)$. As in the proof for $k=2j$, we obtain $u\in X^{2j+1,p}_{\mathcal N,\gamma,R}(\alpha_k-kp,\ldots,\alpha_k)$ such that $\Delta^j_{\gamma}u=\widetilde u$. Therefore, $\phi(u)=\nabla^k_{\gamma}u=v$.
\end{proof}

\subsection{Theorem \ref{theoainbc} for second derivative}

Since $X^{1,p}_{\mathcal N,\gamma,R}=X^{1,p}_{0,R}$ for all $\gamma$, we mention that the Adams-type inequality for weighted Sobolev spaces was solved by \cite[Theorem 1.1]{MR3209335} in the first derivative case. Let us state this inequality for the second derivative case.

\begin{theo}\label{theoadamsk2}
Let $p>1$ and $\theta>-1$. If $\alpha_2-2p+1=0$ and $\gamma>1$, then
\begin{equation*}
\sup_{R\in(0,\infty)}\sup_{\underset{\|\Delta_\gamma u\|_{L^p_{\alpha_2}}\leq1}{u\in X^{2,p}_{R}\cap X^{1,p}_{0,R}}}R^{-(\theta+1)}\int_0^Re^{\mu|u|^{\frac p{p-1}}}r^\theta\mathrm dr<\infty,\mbox{ if }\mu\leq\mu_0,
\end{equation*}
where
\begin{equation*}
\mu_0=(\theta+1)(\gamma-1)^{\frac p{p-1}}.
\end{equation*}
\end{theo}

In order to prove the above Theorem, we need to consider the following Lemma proved by Tarsi (see Lemma 1 in Section 6 of \cite{MR2988207}).

\begin{lemma}\label{lemmatarsi}
Let $p>1$. Then for any $r>0$ there is a constant $C_0=C_0(p,r)$ such that for any positive measurable function $f(s)$ on $(1,\infty)$, satisfying
\begin{equation*}
\int_1^\infty f^ps^{2p-1}\mathrm ds\leq1
\end{equation*}
then
\begin{equation*}
\int_1^\infty e^{rF^q(t)}\dfrac{\mathrm dt}{t^{r+1}}\leq C_0
\end{equation*}
where $\frac1p+\frac1q=1$ and
\begin{equation*}
F(t)=\int_1^t\int_z^\infty f(s)\mathrm ds\mathrm dz.
\end{equation*}
\end{lemma}

\begin{proof}[Proof of Theorem \ref{theoadamsk2}]
Let $u\in X_R^{2,p}\cap X^{1,p}_{0,R}$ such that $\|\Delta_\gamma u\|_{L^p_{\alpha_2}}\leq1$ and $\mu\leq\mu_0$. We can take $v\in AC_{\mathrm{loc}}^2(0,R)$ be the solution of the problem
\begin{equation*}
\left\{\begin{array}{ll}
   -\Delta_{\gamma}v=|\Delta_{\gamma}u|,&\mbox{ in }(0,R)  \\
     v(R)=0,& 
\end{array}\right.
\end{equation*}
given by
\begin{equation*}
v(r)=\int_r^Rs^{-\gamma}\int_0^s|\Delta_{\gamma}u|t^{\gamma}\mathrm dt\mathrm ds\geq0.
\end{equation*}
Lemma \ref{lemmafjs} guarantees that $v$ is well defined with $v\in X^{2,p}_R(\alpha_2-2p,\alpha_2-p,\alpha_2)$. Moreover, $v$ satisfies $|u(t)|\leq v(t)$, for all $t\in(0,R)$. Indeed, by Proposition \ref{prop32} on $u'\in X_R^{1,p}$ we get $r^{\gamma}u'(r)\overset{r\to\infty}\longrightarrow0$. Thus,
\begin{equation*}
v(r)=\int_r^Rs^{-\gamma}\int_0^s|(t^{\gamma}u')'|\mathrm dt\mathrm ds\geq\left|\int_r^Rs^{-\gamma}\int_0^s(t^{\gamma}u')'\mathrm dt\mathrm ds\right|\geq|u(r)|.
\end{equation*}
Now, we consider the following change of variable:
\begin{equation*}
w(t)=v(Rt^{\frac1{1-\gamma}}),\mbox{ for }t\in [1,\infty).
\end{equation*}
It follows
\begin{equation*}
w'(t)=\frac1{1-\gamma}Rt^{\frac{\gamma}{1-\gamma}}v'(Rt^{\frac1{1-\gamma}}),
\end{equation*}
and
\begin{equation*}
w''(t)=\left(\dfrac{R}{1-\gamma}\right)^2t^{\frac{2\gamma}{1-\gamma}}\Delta_{\gamma}v(Rt^{\frac1{1-\gamma}}).
\end{equation*}
Using that $\alpha_2-2p+1=0$ we have
\begin{equation*}
\int_0^R(-\Delta_\gamma v)^pr^{\alpha_2}\mathrm dr=(\gamma-1)^{2p-1}\int_1^{\infty}(-w''(t))^pt^{2p-1}\mathrm dt.
\end{equation*}
Now, since $\|\Delta_{\gamma}v\|_{L^p_{\alpha_2}}=\|\Delta_{\gamma}u\|_{L^p_{\alpha_2}}\leq 1$, we obtain
\begin{equation}\label{joao55a}
\int_1^{\infty}(-(\gamma-1)^{\frac{2p-1}p}w''(t))^pt^{2p-1}\mathrm dt\leq 1.
\end{equation}
By Proposition \ref{prop32} on $v'\in X^{1,p}_R(\alpha_2-p,\alpha_2)$,
\begin{equation*}
w'(\infty)=\lim_{t\to\infty}w'(t)=\dfrac{R^{1-\gamma}}{1-\gamma}\lim_{r\to0}r^\gamma v'(r)=0.
\end{equation*}
Also, from $|u(t)|\leq v(t)$ we get
\begin{align}
\int_0^Re^{\mu|u|^{\frac{p}{p-1}}}r^{\theta}\mathrm dr&\leq \int_0^Re^{\mu_0|u(r)|^{\frac{p}{p-1}}}r^{\theta}\mathrm dr\nonumber\\
&\leq\int_0^Re^{\mu_0v(r)^{\frac{p}{p-1}}}r^{\theta}\mathrm dr\nonumber\\
&=\dfrac{R^{\theta+1}}{\gamma-1}\int_1^{\infty}e^{\mu_0w(t)^{\frac{p}{p-1}}}t^{\frac{\theta+\gamma}{1-\gamma}}\mathrm dt\nonumber\\
&=\dfrac{R^{\theta+1}}{\gamma-1}\int_1^{\infty}\exp\left(\mu_0\dfrac{(\gamma-1)^{-\frac{2p-1}{p-1}}}{(\gamma-1)^{-\frac{2p-1}{p-1}}}w(t)^{\frac p{p-1}}\right)\dfrac{\mathrm dt}{t^{\frac{\theta+\gamma}{\gamma-1}}}\nonumber\\
&=\dfrac{R^{\theta+1}}{\gamma-1}\int_1^{\infty}\exp\left[\dfrac{\theta+1}{\gamma-1}\left((\gamma-1)^{\frac{2p-1}{p}}w(t)\right)^{\frac{p}{p-1}}\right]\dfrac{\mathrm dt}{t^{\frac{\theta+\gamma}{\gamma-1}}},\label{joao56a}
\end{align}
 for any $\mu\leq\mu_0$. Since $w(1)=0$ and $w'(\infty)=0$, we have $w(t)=\int_1^t\int^\infty_z-w''(s)\mathrm ds\mathrm dz$. Therefore, from \eqref{joao55a} and \eqref{joao56a} we can apply Lemma \ref{lemmatarsi} to conclude the proof.
\end{proof}

\subsection{Proof of Adams-type inequality for critical case}

Throughout this subsection, we focus on the critical case of the Theorem \ref{theoainbc} which is equivalent to the next Theorem.

\begin{theo}\label{theo42}
Let $X^{k,p}_{\mathcal N,\gamma,R}(\alpha_0,\ldots,\alpha_k)$ with $\alpha_k-kp+1=0$. Suppose $p>1,\theta>-1$ and $\gamma>k-1$ for $k$ even and $\gamma>k-2$ if $k$ is odd. Then
\begin{equation*}
\sup_{R\in(0,\infty)}\sup_{u\in X^{k,p}_{\mathcal N,\gamma,R},\|\nabla^k_{\gamma}u\|_{L^p_{\alpha_k}}\leq1}R^{-(\theta+1)}\int_0^Re^{\mu|u|^{\frac{p}{p-1}}}r^\theta\mathrm dr<\infty,\quad\forall\mu\leq\mu_0.
\end{equation*}
\end{theo}

Proposition 3.1 in \cite{MR4112674} is similar to our next Lemma besides we suppose $r^\gamma u'(r)\overset{t\to0}\longrightarrow 0$ instead of $\Delta_\gamma u(r)\overset{r\to R}\longrightarrow 0$. For the sake of completeness, we include the proof. 
\begin{lemma}\label{lemmaprop31}
Let $p,q>1$ and $0<R<\infty$. Consider $\gamma$ satisfying $\gamma-2q+1>0$. Then, for any $u\in AC_{\mathrm{loc}}^2(0,R)$ such that $\lim_{r\to R}u(r)=\lim_{r\to 0}r^\gamma u'(r)=0$ we have
\begin{equation*}
\left(\int_0^R|u|^pr^{\frac{(\gamma+1)p}{q^*}-1}\mathrm dr\right)^{\frac1p}\leq C_{\gamma,q}\left(\int_0^R|\Delta_\gamma u|^pr^{\frac{(\gamma+1)p}q-1}\mathrm dr\right)^{\frac1p},
\end{equation*}
where
\begin{equation*}
C_{\gamma,q}=\dfrac{q^2}{(q-1)(\gamma+1)(\gamma-2q+1)}\mbox{ and }q^*=\dfrac{(\gamma+1)q}{\gamma-2q+1}.
\end{equation*}
\end{lemma}
\begin{proof}
Let $w(t)=u(Rt^{\frac{1}{1-\gamma}})$. Then
\begin{equation}\label{gjao}
\int_0^R|u|^pr^{\frac{(\gamma+1)p}{q^*}-1}\mathrm dr=\dfrac{R^{\frac{(\gamma+1)p}{q^*}}}{\gamma-1}\int_1^\infty|w|^pt^{-\frac{(\gamma-2q+1)p}{q(\gamma-1)}-1}\mathrm dt
\end{equation}
and
\begin{equation}\label{gjao2}
\int_0^R|\Delta_\gamma u|^pr^{\frac{(\gamma+1)p}{q}-1}\mathrm dr=(\gamma-1)^{2p}\dfrac{R^{\frac{(\gamma+1)p}{q}-2p}}{\gamma-1}\int_1^\infty|w''(t)|^pt^{p\frac{2q\gamma-\gamma-1}{q(\gamma-1)}-1}\mathrm dt.
\end{equation}
Since $\lim_{r\to R}u(r)=\lim_{r\to 0}r^\gamma u'(r)=0$ we have
\begin{equation*}
w(t)=\int_1^t\int_z^\infty-w''(s)\mathrm ds\mathrm dz.
\end{equation*}
Set
\begin{equation*}
a=\dfrac{p-1}p\dfrac{2q\gamma-\gamma-1}{q(\gamma-1)}.
\end{equation*}
Note that
\begin{align*}
w(t)^p&=\left(\int_1^t\int_z^\infty-w''(s)\dfrac{s^a}{s^a}\mathrm ds\mathrm dz\right)^p\\
&\leq\left(\left[\int_1^t\int_z^\infty|w''(s)|^ps^{ap}\mathrm ds\mathrm dz\right]^{\frac1p}\left[\int_1^t\int_z^\infty s^{-ap'}\mathrm ds\mathrm dz\right]^{\frac{p-1}p}\right)^p\\
&\leq \left[\int_1^t\int_z^\infty|w''(s)|^ps^{ap}\mathrm ds\mathrm dz\right]\left[(\gamma-1)^2C_{\gamma,q}\left(t^{\frac{\gamma-2q+1}{q(\gamma-1)}-1}-1\right)\right]^{p-1}.
\end{align*}
Thus,
\begin{align*}
\int_1^\infty&|w|^pt^{-\frac{(\gamma-2q+1)p}{q(\gamma-1)}-1}\mathrm dt\\
&\leq\int_1^\infty\left[\int_1^t\int_z^\infty|w''(s)|^ps^{ap}\mathrm ds\mathrm dz\right]\left[(\gamma-1)^2C_{\gamma,q}\left(t^{\frac{\gamma-2q+1}{q(\gamma-1)}}-1\right)\right]^{p-1}t^{-\frac{(\gamma-2q+1)p}{q(\gamma-1)}-1}\mathrm dt\\
&=\left((\gamma-1)^2C_{\gamma,q}\right)^{p-1}\int_1^\infty|w''(s)|^ps^{ap}\int_1^s\int_z^\infty t^{-\frac{\gamma-2q+1}{q(\gamma-1)}-1}\mathrm dt\mathrm dz\mathrm ds\\
&\leq\dfrac{\left[(\gamma-1)^2C_{\gamma,q}\right]^{p-1}}{\frac{\gamma-2q+1}{q(\gamma-1)}\left(-\frac{\gamma-2q+1}{q(\gamma-1)}+1\right)}\int_1^\infty|w''(s)|^p\left(s^{-\frac{\gamma-2q+1}{q(\gamma-1)}+1}-1\right)s^{(1-p)\frac{\gamma+1-2q\gamma}{q(\gamma-1)}}\mathrm ds\\
&\leq\left[(\gamma-1)^2C_{\gamma,q}\right]^{p}\int_1^\infty|w''(s)|^ps^{p\frac{2q\gamma-\gamma-1}{q(\gamma-1)}-1}\mathrm ds.
\end{align*}
Using \eqref{gjao} and \eqref{gjao2} we conclude the proof.
\end{proof}
\begin{lemma}\label{lemmahardycons}
Suppose $u\in AC^1_{\mathrm{loc}}(0,R)$ with $\lim_{r\to R}u(r)=0$. If $\alpha-p+1>0$, then
\begin{equation*}
\left(\int_0^R|u|^pr^{\alpha-p}\mathrm dr\right)^{\frac1p}\leq\dfrac{p}{\alpha-p+1}\left(\int_0^R|u'|^pr^\alpha\mathrm dr\right)^{\frac1p}.
\end{equation*}
\end{lemma}
\begin{proof}
By \cite[Theorem 6.2]{MR1069756}, it is enough to check that
\begin{equation*}
\dfrac{p}{\alpha-p+1}=\dfrac{p}{(p-1)^{\frac{p-1}{p}}}\sup_{0<r<R}\|r^{\frac{\alpha-p}p}\|_{L^p(0,r)}\|r^{-\frac{\alpha}p}\|_{L^{p'}(r,R)}.
\end{equation*}
The proof follows because
\begin{equation*}
\|r^{\frac{\alpha-p}p}\|_{L^p(0,r)}\|r^{-\frac{\alpha}p}\|_{L^{p'}(r,R)}=\dfrac{(p-1)^{\frac{p-1}p}}{\alpha-p+1}\left(1-R^{-\frac{\alpha-p+1}{p-1}}r^{(\alpha-p+1)\frac{p}{p-1}}\right)^{\frac{p}{p-1}}.
\end{equation*}
\end{proof}
\begin{lemma}\label{lemmaonetoj}
    Let $p>1,\gamma,\alpha\in\mathbb R$ and $j\geq2$ integer such that $\gamma>(\alpha-p+1)/p$ and $\alpha>2(j-1)p-1$. Suppose $u\in AC_{\mathrm{loc}}^{2j-1}(0,R)$ with $\lim_{r\to R}\Delta^i_{\gamma}u(r)=\lim_{r\to0}r^\gamma(\Delta^i_\gamma u)'(r)=0$ for all $i=1,\ldots,j-1$. Then
    \begin{equation*}
    \|\Delta_{\gamma}u\|_{L^p_{\alpha-2(j-1)p}}\leq\left(\prod_{i=1}^{j-1}C_i\right)\|\Delta_{\gamma}^ju\|_{L^p_{\alpha}},
    \end{equation*}
    where
    \begin{equation*}
    C_i=\dfrac{p^2}{\left[(\gamma+1)p-(\alpha-2(i-1)p+1)\right](\alpha-2ip+1)},\quad\forall i=1,\ldots,j-1.
    \end{equation*}
\end{lemma}
\begin{proof}
Set $\eta_0:=\alpha$. We claim that
\begin{equation}\label{waw}
\|\Delta^{j-i}_{\gamma}u\|_{L^p_{\eta_i}}\leq\widetilde C_i\|\Delta^{j-i+1}_{\gamma}u\|_{L^p_{\eta_{i-1}}}\quad\forall i=1,\ldots,j-1,
\end{equation}
where
\begin{equation*}
q_i=\dfrac{(\gamma+1)p}{\eta_{i-1}+1},\ \widetilde C_i=\dfrac{q^2_i}{(q_i-1)(\gamma+1)(\gamma-2q_i+1)}\mbox{ and }\eta_i=\dfrac{p(\gamma-2q_i+1)}{q_i}-1.
\end{equation*}
Before proving \eqref{waw}, let's check that
\begin{equation}\label{waw1}
\eta_i=\alpha-2ip,\quad\forall i=0,\ldots,j-1.
\end{equation}
For $i=0$, \eqref{waw1} is trivial. Suppose \eqref{waw1} holds for some $i=0,\ldots,j-2$. Then
\begin{align*}
\eta_{i+1}+1&=\dfrac{p(\gamma-2q_{i+1}+1)}{q_{i+1}}=\dfrac{p\left[(\gamma+1)(\eta_i+1)-2q_{i+1}(\eta_i+1)\right]}{q_{i+1}(\eta_i+1)}\\
&=\dfrac{p\left[(\gamma+1)(\alpha-2ip+1)-2(\gamma+1)p\right]}{p(\gamma+1)}=\alpha-2(i+1)p+1.
\end{align*}
Thus \eqref{waw1} follows.

Since $\gamma>(\alpha-p+1)/p$ and $\alpha>2(j-1)p-1$, we obtian $q_i>1$, $\eta_i+1>0$ and $\gamma-2q_i+1>0$ for all $i=1,\ldots,j-1$. Then, we can use Lemma \ref{lemmaprop31} to obtain \eqref{waw}.

Therefore, applying \eqref{waw} we prove the Lemma if $\widetilde C_i=C_i$. By \eqref{waw1},
\begin{align*}
\widetilde C_i\!&=\!\dfrac{(\eta_{i-1}+1)^2q_i^2}{\left[q_i(\eta_{i-1}+1)-\eta_{i-1}-1\right](\gamma+1)\left[(\gamma+1)(\eta_{i-1}+1)-2q_i(\eta_{i-1}+1)\right]}\\
&=\dfrac{(\gamma+1)^2p^2}{\left[(\gamma+1)p-(\alpha-2(i-1)p+1)\right](\gamma+1)^2\left[\alpha-2(i-1)p+1-2p\right]}\\
&=C_i.
\end{align*}
\end{proof}

\begin{proof}[Proof of the Theorem \ref{theo42}]
Consider $u\in X^{k,p}_{\mathcal N,\gamma,R}$ with $\|\nabla^k_{\gamma}u\|_{L^p_{\alpha_k}}\leq1$. We have divided our proof into two cases:

\noindent\textbf{Case $k=2j$:} We can suppose $j\geq2$ beacause the case $k=2$ was solved in Theorem \ref{theoadamsk2}. Since $\gamma>k-1$, $\alpha_k-kp+1=0$, $u\in X^{k,p}_{\mathcal N,\gamma,R}$ and $r^\gamma\nabla^i_\gamma u(r)\overset{r\to0}{\longrightarrow}0$ for all $i=1,\ldots,k-1$ (see equation \eqref{eqclaimhsu}), Lemma \ref{lemmaonetoj} guarantees
\begin{equation}\label{waw4}
\|\Delta_{\gamma}u\|_{L^p_{2p-1}}\leq\left(\prod_{i=1}^{j-1}C_i\right)\|\Delta^j_{\gamma}u\|_{L^p_{\alpha_k}},
\end{equation}
with
\begin{equation*}
C_i=\dfrac1{2^2\left(\frac{\gamma+1}2-j+i-1\right)(j-i)},\quad\forall i=1,\ldots,j-1.
\end{equation*}
Define $v=u/\prod_{i=1}^{j-1}C_i$. Then, by \eqref{waw4},
\begin{equation*}
\|\Delta_{\gamma}v\|_{L^p_{2p-1}}=\dfrac{\|\Delta_{\gamma} u\|_{L^p_{2p-1}}}{\prod_{i=1}^{j-1}C_i}\leq\|\Delta^j_{\gamma}u\|_{L^p_{\alpha_k}}=\|\nabla^k_{\gamma}u\|_{L^p_{\alpha_k}}\leq1.
\end{equation*}
Using Theorem \ref{theoadamsk2} we obtain $C=C(\gamma,\theta,p)>0$ such that
\begin{equation*}
C\geq R^{-(\theta+1)}\int_0^Re^{(\theta+1)(\gamma-1)^{p'}|v|^{p'}}r^\theta\mathrm dr.
\end{equation*}
Thus,
\begin{equation*}
R^{-(\theta+1)}\int_0^R\exp\left[(\theta+1)\left(\dfrac{\gamma-1}{\prod_{i=1}^{j-1}C_i}\right)^{p'}|u|^{p'}\right]r^\theta\mathrm dr\leq C.
\end{equation*}
The case $k=2j$ follows if we show that
\begin{equation*}
(\gamma-1)\prod_{i=1}^{j-1}C_i^{-1}=2^{k-1}\dfrac{\Gamma(\frac{k}2)\Gamma(\frac{\gamma+1}2)}{\Gamma(\frac{\gamma+1-k}2)}.
\end{equation*}
Indeed,
\begin{align*}
(\gamma-1)\prod_{i=1}^{j-1}C_i^{-1}&=(\gamma-1)\prod_{i=1}^{j-1}2^2\left(\frac{\gamma+1}2-j+i-1\right)(j-i)\\
&=(\gamma-1)2^{k-2}(j-1)!\left(\frac{\gamma+1}2-j\right)\ldots\left(\frac{\gamma+1}2-2\right)\\
&=\dfrac{\gamma-1}{2}2^{k-1}\Gamma(j)\dfrac{\Gamma(\frac{\gamma+1}2-1)}{\Gamma(\frac{\gamma+1}2-j)}\\
&=2^{k-1}\dfrac{\Gamma(\frac{k}2)\Gamma(\frac{\gamma+1}2)}{\Gamma(\frac{\gamma+1-k}2)}.
\end{align*}

\noindent\textbf{Case $k=2j+1$:} For now let us suppose $k>3$. Since $\gamma>k-2$, $\alpha_k-kp+1=0$, $u\in X^{k,p}_{\mathcal N,\gamma,R}$ and $r^\gamma\nabla^i_\gamma u(r)\overset{r\to0}{\longrightarrow}0$ for all $i=1,\ldots,k-2$ (apply equation \eqref{eqclaimhsu} on $u\in X^{k-1,p}_R(\alpha_0,\ldots,\alpha_{k-2},\alpha_k-p)$), Lemma \ref{lemmaonetoj} guarantees
\begin{equation}\label{waw5}
\|\Delta_{\gamma}u\|_{L^p_{2p-1}}\leq\left(\prod_{i=1}^{j-1}C_i\right)\|\Delta^j_{\gamma}u\|_{L^p_{\alpha_k-p}},
\end{equation}
where
\begin{equation*}
C_i=\dfrac1{2^2\left(\frac{\gamma+1}2-j+i-1\right)(j-i)},\quad\forall i=1,\ldots j-1.
\end{equation*}
Define $v=u/\frac{p}{\alpha_k-p+1}\prod_{i=1}^{j-1}C_i$ for $k>3$ and $v=u/\frac{p}{\alpha_k-p+1}$ for $k=3$. Then, by \eqref{waw5} and Lemma \ref{lemmahardycons},
\begin{equation*}
\|\Delta_\gamma v\|_{L^p_{2p-1}}\leq\dfrac{\|\Delta^j_{\gamma} u\|_{L^p_{\alpha_k-p}}}{\frac{p}{\alpha_k-p+1}}\leq\|(\Delta^j_{\gamma}u)'\|_{L^p_{\alpha_k}}=\|\nabla^k_{\gamma}u\|_{L^p_{\alpha_k}}\leq1.
\end{equation*}
Using Theorem \ref{theoadamsk2} we have $C=C(\gamma,\theta,p)>0$ such that
\begin{equation*}
C\geq R^{-(\theta+1)}\int_0^Re^{(\theta+1)(\gamma-1)^{p'}|v|^{p'}}r^{\theta}\mathrm dr.
\end{equation*}
Then (for simplicity denote $\prod_{i=1}^{j-1}C_i=1$ if $k=3$), 
\begin{equation*}
R^{-(\theta+1)}\int_0^R\exp\left[(\theta+1)\left(\dfrac{\gamma-1}{\frac{p}{\alpha_k-p+1}\prod_{i=1}^{j-1}C_i}\right)^{p'}|u|^{p'}\right]r^\theta\mathrm dr\leq C.
\end{equation*}
The case $k=2j+1$ follows if we show that
\begin{equation*}
(\gamma-1)\dfrac{\alpha_k-p+1}{p}\prod_{i=1}^{j-1}C_i^{-1}=2^{k-1}\dfrac{\Gamma(\frac{k+1}2)\Gamma(\frac{\gamma+1}2)}{\Gamma(\frac{\gamma+2-k}2)}.
\end{equation*}
Indeed,
\begin{align*}
(\gamma-1)\dfrac{\alpha_k-p+1}{p}\prod_{i=1}^{j-1}C_i^{-1}&=(\gamma-1)2j\prod_{i=1}^{j-1}2^2\left(\dfrac{\gamma+1}2-j+i-1\right)(j-i)\\
&=\dfrac{\gamma-1}22^{2j}j(j-1)!\left(\dfrac{\gamma+1}2-j\right)\cdots\left(\dfrac{\gamma+1}2-2\right)\\
&=\dfrac{\gamma-1}22^{k-1}\Gamma(j+1)\dfrac{\Gamma(\frac{\gamma+1}2-1)}{\Gamma(\frac{\gamma+1}2-j)}\\
&=2^{k-1}\dfrac{\Gamma(\frac{k+1}2)\Gamma(\frac{\gamma+1}2)}{\Gamma({\frac{\gamma+2-k}2})}.
\end{align*}
\end{proof}

\subsection{Proof the supremum is unbounded for supercritical case}

In view of Theorem \ref{theo42}, to conclude the proof of Theorem \ref{theoainbc} we are left with the task to show the following Theorem.

\begin{theo}
Let $X^{k,p}_{\mathcal N,\gamma,R}(\alpha_0,\ldots,\alpha_k)$ with $\alpha_k-kp+1=0$. Suppose $p>1,\theta>-1$ and $\gamma>k-1$ for $k$ even and $\gamma>k-2$ if $k$ is odd. If $R\in(0,\infty)$ and $\mu>\mu_0$ we have
\begin{equation*}
\sup_{u\in X^{k,p}_{\mathcal N,\gamma,R},\|\nabla^k_{\gamma}u\|_{L^p_{\alpha_k}}\leq1}\int_0^Re^{\mu|u|^{\frac{p}{p-1}}}r^\theta\mathrm dr=\infty.
\end{equation*}
\end{theo}
\begin{proof}
The construction of the sequence is a similar argument as used in the proof of Proposition \ref{prop54}. Let $\phi\in C^\infty[0,1]$ such that
\begin{equation*}
\left\{\begin{array}{ll}
\phi(0)=\phi'(0)=\cdots=\phi^{(k-1)}(0)=0,\ \phi'\geq0,\\
\phi(1)=\phi'(1)=1,\ \phi''(1)=\cdots\phi^{(k-1)}(1)=0.
\end{array}\right.
\end{equation*}
Consider $0<\varepsilon<1/2$ and
\begin{equation*}
H(t)=\left\{\begin{array}{llll}
     \varepsilon\phi\left(\dfrac{t}{\varepsilon}\right),&\mbox{if }0<t\leq\varepsilon  \\
     t,&\mbox{if }\varepsilon<t\leq1-\varepsilon,\\
     1-\varepsilon\phi\left(\dfrac{1-t}{\varepsilon}\right),&\mbox{if }1-\varepsilon<t\leq1,\\
     1,&\mbox{if }t>1.
\end{array}\right.
\end{equation*}
Let $m\in\mathbb N$ and the sequence $(\psi_{m,\varepsilon})_m$ given by
\begin{equation*}
\psi_{m,\varepsilon}(r)=H\left(\dfrac{\log\frac{R}{r}}{\log m}\right),\quad r>0.
\end{equation*}
It is not hard to see that $\psi_{m,\varepsilon}\in X^{k,p}_{0,R}\subset X^{k,p}_{\mathcal N,R}$. By induction on $n\in\mathbb N$ we have
\begin{equation}\label{eqharmonicpsi}
\Delta^n_{\gamma}\psi_{m,\varepsilon}(r)=\dfrac1{r^{2n}}\sum_{i=1}^{2n}\dfrac{c_{in}}{(\log m)^i}H^{(i)}\left(\dfrac{\log\frac{R}{r}}{\log m}\right),
\end{equation}
where
\begin{equation*}
\left\{\begin{array}{l}
c_{11}=-(\gamma-1),\ c_{21}=-1;\\
c_{1n+1}=2n\left(\gamma-2n-1\right)c_{1n};  \\
     c_{2n+1}=-2n(2n+1)c_{2n}-(4n+1)c_{1n}+2n\gamma c_{2n}+\gamma c_{1n};\\
     c_{2n+1n+1}=-(4n+1)c_{2nn}-c_{2n-1n}+\gamma c_{2nn}\\
     c_{2n+2n+1}=-c_{2nn}.
\end{array}\right.
\end{equation*}
and, for each $i=3,\ldots,2n$,
\begin{equation*}
c_{in+1}=-2n(2n+1)c_{in}-(4n+1)c_{i-1n}-c_{i-2n}+2n\gamma c_{in}+\gamma c_{i-1n}.
\end{equation*}
From $c_{11}=-(\gamma-1)$ and $c_{1n+1}=2n(\gamma-2n-1)c_{1n}$ we obtain
\begin{align}
c_{1n}&=-(\gamma-1)2^{n-1}(n-1)!\prod_{i=1}^{n-1}(\gamma-2i-1)\nonumber\\
&=-(\gamma-1)2^{2n-2}(n-1)!\prod_{i=1}^{n-1}\left(\dfrac{\gamma-1}2-i\right)\nonumber\\
&=-\dfrac{\gamma-1}{2}2^{2n-1}(n-1)!\dfrac{\Gamma\left(\frac{\gamma-1}2\right)}{\Gamma\left(\frac{\gamma-1}2-(n-1)\right)}\nonumber\\
&=-2^{2n-1}\dfrac{\Gamma(n)\Gamma\left(\frac{\gamma+1}2\right)}{\Gamma\left(\frac{\gamma+1-2n}2\right)}.\label{eqc1n}
\end{align}

\noindent\textbf{Case $k=2j$:} By \eqref{eqharmonicpsi},
\begin{equation*}
\Delta^j_\gamma\psi_{m,\varepsilon}(r)=\dfrac{c_{1j}}{r^{2j}\log m}H'\left(\dfrac{\log\frac{R}{r}}{\log m}\right)+\dfrac{O\left((\log m)^{-2}\right)}{r^{2j}}.
\end{equation*}
Then
\begin{align}
\|\nabla_\gamma^k\psi_{m,\varepsilon}\|^p_{L^p_{\alpha_k}}&=\left|\dfrac{c_{1j}}{\log m}\right|^p\int_{\frac{R}m}^R\left|H'\left(\dfrac{\log\frac{R}{r}}{\log m}\right)+O\left((\log m)^{-1}\right)\right|^pr^{\alpha_k-kp}\mathrm dr\nonumber\\
&\leq\left|\dfrac{c_{1j}}{\log m}\right|^p\Bigg[\int_{\frac{R}{m}}^{Rm^{\varepsilon-1}}\left|\|\phi'\|_\infty+O\left((\log m)^{-1}\right)\right|^pr^{-1}\mathrm dr\nonumber\\
&\quad+\int_{Rm^{\varepsilon-1}}^{Rm^{-\varepsilon}}r^{-1}\mathrm dr+\int_{Rm^{-\varepsilon}}^R\left|\|\phi'\|_{\infty}+O\left((\log m)^{-1}\right)\right|^pr^{-1}\Bigg]\nonumber\\
&\leq|c_{1j}|^p(\log m)^{1-p}\left[1+2\varepsilon\left|\|\phi'\|_\infty+O\left((\log m)^{-1}\right)\right|^p\right].\label{eqbpsi}
\end{align}
Set
\begin{equation*}
u_{m,\varepsilon}(r)=\dfrac{\psi_{m,\varepsilon}(r)}{\|\nabla^k_\gamma\psi_{m,\varepsilon}\|_{L ^p_{\alpha_k}}}.
\end{equation*}
From \eqref{eqc1n} we obtain $\mu_0=(\theta+1)|c_{1j}|^{\frac{p}{p-1}}$. Given $\mu>\mu_0$, \eqref{eqbpsi} guarantees
\begin{align*}
\int_0^Re^{\mu u_{m,\varepsilon}^{p'}}r^\theta\mathrm dr&\geq\int_0^{\frac{R}m}e^{\mu\|\nabla^k_\gamma\psi_{m,\varepsilon}\|_{L^p_{\alpha_k}}^{-p'}}r^\theta\mathrm dr\\
&\geq e^{\frac{\mu\log m}{|c_{1j}|^{p'}[1+2\varepsilon\|\phi'\|_\infty+O\left((\log m)^{-1}\right)]^{1/(p-1)}}-(\theta+1)\log m}\dfrac{R^{\theta+1}}{\theta+1}\\
&=\dfrac{R^{\theta+1}}{\theta+1}e^{(\theta+1)\log m\left(\frac{\mu}{\mu_0}\frac{1}{[1+2\varepsilon\|\phi'\|_\infty+O\left((\log m)^{-1}\right)]^{1/(p-1)}}-1\right)}.
\end{align*}
Therefore, taking $\varepsilon>0$ sufficiently small we have that the right term tends to be infinite when $m\to\infty$.

\noindent\textbf{Case $k=2j+1$:} Derivate \eqref{eqharmonicpsi} to obtain
\begin{equation*}
\nabla^k_\gamma\psi_{m,\varepsilon}(r)=-\dfrac{2jc_{1j}}{r^{2j+1}\log m}H'\left(\dfrac{\log\frac{R}{r}}{\log m}\right)+\dfrac{O\left((\log m)^{-2}\right)}{r^{2j+1}}
\end{equation*}
Proceeding as in the previous case,
\begin{equation}\label{eqbpsi2}
\|\nabla_\gamma^k\psi_{m,\varepsilon}\|^p_{L^p_{\alpha_k}}\leq |2jc_{1j}|^p(\log m)^{1-p}\left[1+2\varepsilon\left|\|\psi'\|_\infty+O\left((\log m)^{-1}\right)\right|^p\right].
\end{equation}
Set
\begin{equation*}
u_{m,\varepsilon}(r)=\dfrac{\psi_{m,\varepsilon}(r)}{\|\nabla^k_\gamma\psi_{m,\varepsilon}\|_{L ^p_{\alpha_k}}}.
\end{equation*}
From \eqref{eqc1n} we obtain $\mu_0=(\theta+1)|2jc_{1j}|^{\frac{p}{p-1}}$. Given $\mu>\mu_0$, \eqref{eqbpsi2} guarantees
\begin{align*}
\int_0^Re^{\mu u_{m,\varepsilon}^{p'}}r^\theta\mathrm dr&\geq\int_0^{\frac{R}m}e^{\mu\|\nabla^k_\gamma\psi_{m,\varepsilon}\|_{L^p_{\alpha_k}}^{-p'}}r^\theta\mathrm dr\\
&\geq e^{\frac{\mu\log m}{|2jc_{1j}|^{p'}[1+2\varepsilon\|\phi'\|_\infty+O\left((\log m)^{-1}\right)]^{1/(p-1)}}-(\theta+1)\log m}\dfrac{R^{\theta+1}}{\theta+1}\\
&=\dfrac{R^{\theta+1}}{\theta+1}e^{(\theta+1)\log m\left(\frac{\mu}{\mu_0}\frac{1}{[1+2\varepsilon\|\phi'\|_\infty+O\left((\log m)^{-1}\right)]^{1/(p-1)})}-1\right)}.
\end{align*}
Therefore, taking $\varepsilon>0$ sufficiently small we have that the right term tends to be infinite when $m\to\infty$.
\end{proof}

\section{Applications in a PDE}\label{APDE}

\subsection{Application for Sobolev case}

Define the $\alpha$-generalized radial Laplace operator $\Delta_\alpha u:=-r^{-\alpha}(r^{\alpha}u')'$ with $\alpha>3$. Consider the following problem
\begin{equation}\label{problem1}
\left\{\begin{array}{ll}
     \Delta_\alpha^2u=r^{\theta-\alpha}f(r,u)&\mbox{in }(0,R),  \\
     u=\Delta_\alpha u=0&\mbox{in }R,\\
     u'=(\Delta_\alpha u)'=0&\mbox{in }0,
\end{array}\right.
\end{equation} 
with $\theta\geq\alpha-4$ and $f\colon [0,R]\times\mathbb R\to\mathbb R$ continuous satisfying
\begin{equation}\tag{$f_S$}\label{eqfS}
|f(r,t)|\leq c_1+c_2|u|^{p-1},
\end{equation}
for some $c_1,c_2>0$ and $1\leq p\leq 2(\theta+1)/(\alpha-3)$. Also, we are considering the following space
\begin{equation*}
X_{S}:=X^{2,2}_R(\alpha-4,\alpha-2,\alpha)\cap X^{1,2}_{0,R}(\alpha-4,\alpha-2).
\end{equation*}
We say that $u\in X_S$ is a \textit{weak solution of \eqref{problem1}} if
\begin{equation*}
\int_0^R\Delta_\alpha u\Delta_\alpha vr^\alpha\mathrm dr=\int_0^Rf(r,u)vr^{\theta}\mathrm dr\quad\forall v\in X_S.
\end{equation*}
The energy functional $J\colon X_S\to\mathbb R$ is given by
\begin{equation*}
J(u)=\dfrac12\int_0^R|\Delta_\alpha u|^2r^\alpha\mathrm dr-\int_0^RF(r,u)r^\theta\mathrm dr,
\end{equation*}
where $F(r,t)=\int_0^tf(r,s)\mathrm ds$. Note that \eqref{eqfS} guarantees
\begin{equation}\tag{$F_S$}\label{eqFS}
|F(r,t)|\leq c_1|u|+c_3|u|^p,
\end{equation}
for some $c_3>0$.
\begin{cor}
The norm
\begin{equation*}
\|u\|_{\Delta_\alpha}:=\left(\int_0^R|\Delta_\alpha u|^2r^\alpha\mathrm dr\right)^{\frac12}
\end{equation*}
is equivalent to $\|\cdot\|_{X^{2,2}_R}$ in $X_S$.
\end{cor}
\begin{proof}
Follows directly from Proposition \ref{propequivnormL}.
\end{proof}

\begin{prop}\label{prop42}
Suppose $1<p<2(\theta+1)/(\alpha-3)$ and $f(r,u)=g(r)|u|^{p-2}u$ for some $g\colon[0,R]\to\mathbb R$ continuous positive function. Set
\begin{equation*}
m_{\alpha,\theta}:=\inf_{u\in X_S\backslash\{0\}}\dfrac{\int_0^R|\Delta_\alpha u|^2r^\alpha\mathrm dr}{\left(\int_0^Rg(r)|u|^{p}r^{\theta}\mathrm dr\right)^{\frac2{p}}}.
\end{equation*}
Then $m_{\alpha,\theta}$ is attained by a function $u_0\in X_S\backslash\{0\}$. Moreover, up to a suitable multiple, $u_0$ is a weak solution of \eqref{problem1}.
\end{prop}
\begin{proof}
Let $(u_n)$ be a minimizing sequence in $X_S$ of $m_{\alpha,\theta}$. We can suppose $\|u_n\|_{\Delta_\alpha}=1$ changing $u_n$ for $u_n/\|u_n\|_{\Delta_\alpha}$ if necessary. Since $X_S$ is reflexive, up to subsequence, $u_n\rightharpoonup u_0$ for some $u_0\in X_S$. The compact embedding of Theorem \ref{theo32} implies 
\begin{equation*}
\int_0^Rg(r)|u_n|^pr^\theta\mathrm dr\to\int_0^Rg(r)|u_0|^pr^\theta\mathrm dr.
\end{equation*}
Note that $u_0\in X_S\backslash\{0\}$, because otherwise
\begin{equation*}
1=\left(\dfrac{\int_0^Rg(r)|u_n|^pr^\theta\mathrm dr}{\int_0^Rg(r)|u_n|^pr^\theta\mathrm dr}\right)^{\frac{2}{p}}\|u_n\|_{\Delta_\alpha}^2\to\left(\int_0^Rg(r)|u_0|^pr^\theta\mathrm dr\right)^{\frac{2}{p}}m_{\alpha,\theta}=0.
\end{equation*}
Since $(u_n)_n$ is a minimizing sequence and $u_n\rightharpoonup u_0$ in $X_S$,
\begin{equation*}
m_{\alpha,\theta}\leq\dfrac{\|u_0\|_{\Delta_\alpha}^2}{\|g(r)^{1/p}u_0\|_{L^p_\theta}^2}\leq\dfrac{1}{\lim\|g(r)^{1/p}u_n\|_{L^p_\theta}^2}=\lim\dfrac{\|u_n\|_{\Delta_\alpha}^2}{\|g(r)^{1/p}u_n\|_{L^p_\theta}^2}=m_{\alpha,\theta},
\end{equation*}
which concludes that $u_0$ attains $m_{\alpha,\theta}$.

By Lagrange's multipliers, there exists $\lambda>0$ such that
\begin{equation*}
\int_0^R\Delta_\alpha u_0\Delta_\alpha vr^\alpha\mathrm dr=\lambda\int_0^Rg(r)|u_0|^{p-2}u_0vr^\theta\mathrm dr,\quad\forall v\in X_S.
\end{equation*}
Note that $\widetilde u_0(r)=\lambda^{\frac{1}{p-1}}u_0(r)$ is a weak solution of \eqref{problem1}.
\end{proof}

\begin{prop}\label{propclassicsolution1}
Suppose $u_0$ is a weak solution of \eqref{problem1} with $p\geq2$. Then $u_0\in C^4(0,R]$, $\Delta_\alpha u_0\in C^2(0,R]$ and $\Delta_\alpha ^2u_0=r^{\theta-\alpha}f(r,u_0)$ $\forall r\in(0,R]$. Moreover, $u_0(R)=\Delta_\alpha u_0(R)=0$.
\end{prop}
\begin{proof}
We claim that $\Delta_\alpha u_0$ has weak derivative given by
\begin{equation}\label{wdlu}
(\Delta_\alpha u_0)'(r)=-r^{-\alpha}\int_0^rf(s,u_0(s))s^\theta\mathrm ds.
\end{equation}
Note that the right term is well defined by \eqref{eqfS} and Theorem \ref{theo32} with $p\geq2$. Using \eqref{wdlu} and that the right term of \eqref{wdlu} is $C^1((0,R])$, we are left with the task to prove \eqref{wdlu} and $\Delta_\alpha u_0(R)=0$.

Let $\varphi\in C^\infty_0(0,R)$. Writing $v(r)=-\int_r^R\varphi s^{-\alpha}\mathrm ds$ we have $\varphi=r^\alpha v'$ and $v\in X_S$. Using that $u_0$ is a weak solution and Fubini's Theorem we get
\begin{align*}
\int_0^R\Delta_\alpha u_0\varphi'\mathrm dr&=-\int_0^R\Delta_\alpha u_0\Delta_\alpha vr^\alpha\mathrm dr\\
&=\int_0^R\int_r^Rf(r,u_0(r))\varphi(s)s^{-\alpha}r^{\theta}\mathrm ds\mathrm dr\\
&=\int_0^Rs^{-\alpha}\int_0^sf(r,u_0(r))r^\theta\mathrm dr\varphi(s)\mathrm ds.
\end{align*}
This concludes \eqref{wdlu}.

Let us check that $\Delta_\alpha u_0(R)=0$. Fix $v\colon[0,R]\to\mathbb R$ smooth such that $v\equiv0$ in $[0,R/3]$ and $v(r)=R-r$ in $[2R/3,R]$. Note that $v\in X_S$ and by $u_0$ be a weak solution of \eqref{problem1} we have
\begin{align*}
\int_0^Rf(r,u_0)vr^\theta\mathrm dr&=\int_0^R\Delta_\alpha u_0\Delta_\alpha vr^\alpha\mathrm dr=-\int_0^R\Delta_\alpha u_0(r^\alpha v')'\mathrm dr\\
&=-\Delta_\alpha u_0(R)v'(R)R^\alpha+\int_0^Rr^\alpha(\Delta_\alpha u_0)'v'\mathrm dr\\
&=\Delta_\alpha u_0(R)R^\alpha-\int_0^R\left(r^\alpha(\Delta_\alpha u_0)'\right)'v\mathrm dr\\
&=\Delta_\alpha u_0(R)R^\alpha+\int_0^R\Delta_\alpha ^2u_0vr^\alpha\mathrm dr.
\end{align*}
Since $\Delta_\alpha ^2u_0=r^{\theta-\alpha}f(r,u_0)$ we conclude that $\Delta_\alpha u_0(R)=0$.
\end{proof}

\begin{lemma}\label{lemmabootstrap1}
Suppose $2\leq p<2(\theta+1)/(\alpha-3)$ and $\theta\geq\alpha-2$. If $u_0$ is a weak solution of \eqref{problem1}, then $u_0\in X^{2,s}_R(\alpha-2s,\alpha-s,\alpha)$ for all $s\geq1$.
\end{lemma}
\begin{proof}
By Proposition \ref{propclassicsolution1} we have that $\Delta_\alpha ^2u_0=r^{\theta-\alpha}f(r,u_0)$ with $u_0(R)=\Delta_\alpha u_0(R)=0$. Let us check that, given $s\geq1$,
\begin{equation}\label{equivlx2r}
u_0\in L^{s(p-1)}_{\alpha+s(\theta-\alpha+2)}\Rightarrow u_0\in X^{2,s}_R(\alpha-2s,\alpha-s,\alpha).
\end{equation}
Indeed, from $u_0\in L^{s(p-1)}_{\alpha+s(\theta-\alpha+2)}$ and \eqref{eqfS} we obtain $\Delta_\alpha ^2u_0=r^{\theta-\alpha}f(r,u_0)\in L^s_{\alpha+2s}$. Since $r^\alpha \Delta_\alpha u_0\overset{r\to0}\longrightarrow 0$ (see \eqref{wdlu}) and $\Delta_\alpha u_0(R)=0$ we can write
\begin{equation*}
\Delta_\alpha u_0(r)=\int_r^Rt^{-\alpha}\int_0^t\Delta_\alpha ^2u_0(s)s^\alpha\mathrm ds\mathrm dt.
\end{equation*}
Lemma \ref{lemmafjs} guarantees $\Delta_\alpha u_0\in X^{2,s}_R(\alpha,\alpha+s,\alpha+2s)$. Since $r^\alpha u_0\overset{r\to0}{\longrightarrow}0$ and $u_0(R)=0$ we can also write
\begin{equation*}
u_0(r)=\int_r^Rt^{-\alpha}\int_0^t\Delta_\alpha u_0(s)s^\alpha\mathrm ds\mathrm dt.
\end{equation*}
Therefore, Lemma \ref{lemmafjs} concludes \eqref{equivlx2r}.

Let $q\geq1$ such that $q\left[(p-1)(\alpha-3)-2(\theta-\alpha+2)\right]\leq 2(\alpha+1)$. Applying Theorem \ref{theo32} in $u_0\in X^{2,2}_R(\alpha-4,\alpha-2,\alpha)$ we get $u_0\in L^{q(p-1)}_{\alpha+2q}$. Note that \eqref{equivlx2r} implies that $u_0\in X^{2,q}_R(\alpha-2q,\alpha-q,\alpha)$. Thus, we only need to consider the case when
\begin{equation}\label{a1positive}
(p-1)(\alpha-3)-2(\theta-\alpha+2)>0.
\end{equation}
Denote $q:=2(\alpha+1)/[(p-1)(\alpha-3)-2(\theta-\alpha+2)]$. By $p<2(\theta+1)/(\alpha-3)$ we have $q>2$. We can suppose
\begin{equation}\label{a2positive}
(p-1)\left[(p-1)(\alpha-3)-2(\theta-\alpha+2)\right]-2[\theta-\alpha+2+2(p-1)]>0.
\end{equation}
Indeed, if \eqref{a2positive} does not hold, then $u_0\in X^{2,q}_R(\alpha-2q,\alpha-q,\alpha)$ implies that $u_0\in L^{s(p-1)}_{\alpha+s(\theta-\alpha+2)}$ for all $s\geq1$. Then \eqref{equivlx2r} guarantees $u_0\in X^{2,s}_R(\alpha-2s,\alpha-2,\alpha)$ for all $s\geq1$.  

Set
\begin{equation*}
k:=\dfrac{(p-1)(\alpha-3)-2(\theta-\alpha+2)}{(p-1)\left[(p-1)(\alpha-3)-2(\theta-\alpha+2)\right]-2\left[\theta-\alpha+2+2(p-1)\right]}.
\end{equation*}
Using \eqref{a1positive} and \eqref{a2positive} we get that $k$ is well defined and it is positive. Moreover, by $p<2(\theta+1)/(\alpha-3)$,
\begin{align*}
    k&=\dfrac{\alpha-3-\frac{2(\theta-\alpha+2)}{p-1}}{(p-1)(\alpha-3)-2(\theta-\alpha+2)-4-\frac{2(\theta-\alpha+2)}{p-1}}\\
    &>\dfrac{\alpha-3-\frac{2(\theta-\alpha+2)}{p-1}}{(\frac{2(\theta+1)}{\alpha-3}-1)(\alpha-3)-2(\theta-\alpha+2)-4-\frac{2(\theta-\alpha+2)}{p-1}}=1.
\end{align*}
The proof is complete by showing that
\begin{equation}\label{bootstrap1}
u_0\in X^{2,s}_R(\alpha-2s,\alpha-s,\alpha),\ s\geq q\Rightarrow u_0\in X^{2,ks}_R(\alpha-2ks,\alpha-ks,\alpha).
\end{equation}
Firstly, we can suppose $\alpha-2s+1>0$ or otherwise the Adams-Trudinger-Moser case of Theorem \ref{theo32} concludes. Note that
\begin{equation*}
q=\dfrac{2(\alpha+1)}{(p-1)(\alpha-3)-2(\theta-\alpha+2)}=\dfrac{(p-1)(\alpha+1)-(\alpha+1)k^{-1}}{\theta-\alpha+2+2(p-1)}.
\end{equation*}
Then
\begin{align*}
s\geq q&\Rightarrow s\left[k(\theta-\alpha+2)+2k(p-1)\right]\geq k(p-1)(\alpha+1)-(\alpha+1)\\
&\Rightarrow ks(p-1)\leq\dfrac{s\left[\alpha+1+ks(\theta-\alpha+2)\right]}{\alpha-2s+1}.
\end{align*}
Using $u_0\in X^{2,s}_R(\alpha-2s,\alpha-s,\alpha)$ we get $u_0\in L^{ks(p-1)}_{\alpha+ks(\theta-\alpha+2)}$. Therefore, \eqref{bootstrap1} follows by \eqref{equivlx2r}.
\end{proof}

\begin{prop}\label{propclsol1}
Suppose $u_0$ is a weak solution of \eqref{problem1} with $\theta>\alpha-1$ and $2\leq p<2(\theta+1)/(\alpha-3)$. Then $u_0\in C^4((0,R])\cap C^3([0,R])$ is a classical solution of \eqref{problem1} with $\Delta_\alpha u_0\in C^2((0,R])\cap C^1([0,R])$ and $u'_0(0)=(\Delta_\alpha u_0)'(0)=u_0(R)=\Delta_\alpha u_0(R)=0$. Moreover, $u''_0(0)=-\Delta_\alpha u_0(0)/(\alpha+1)$ and $u'''_0(0)=0$.
\end{prop}
\begin{proof}
From Proposition \ref{propclassicsolution1}, Lemma \ref{lemmabootstrap1} and Morrey case of Theorem \ref{theo32} we have $u_0\in C^4((0,R])\cap C^1([0,R])$ with $\Delta_\alpha u_0\in C^2((0,R])$ and $\Delta_\alpha ^2u_0=r^{\theta-\alpha}f(r,u_0)$ for all $r\in(0,R]$. Let us check that $u_0'(0)=(\Delta_\alpha u_0)'(0)=0$. From L'Hopital's rule in \eqref{wdlu} with $\theta>\alpha-1$ we have
\begin{equation}\label{ganojdfna}
\lim_{r\to0}(\Delta_\alpha u_0)'(r)=-\lim_{r\to0}\dfrac{f(r,u_0(r))r^\theta}{\alpha r^{\alpha-1}}=0.
\end{equation}
Using $u_0'(r)=-r^{-\alpha}\int_0^r\Delta_\alpha u_0(s)s^\alpha\mathrm ds$ and L'Hopital's rule twice we obtain
\begin{equation*}
\lim_{r\to0}u'_0(r)=-\alpha^{-1}\lim_{r\to0}\Delta_\alpha u_0(r)r=\alpha^{-1}\lim_{r\to0}(\Delta_\alpha u_0)'(r)r^2=0.
\end{equation*}
By \eqref{wdlu} and \eqref{ganojdfna} we have $\Delta_\alpha u_0\in C^1([0,R])$. 

Now we only need to show that $u_0\in C^3([0,R])$ with $u_0''(0)=- \Delta_\alpha u_0(0)/(\alpha+1)$ and $u_0'''(0)=0$. Note that
\begin{equation*}
u_0''(r)=-\alpha\dfrac{u_0'(r)}r-\Delta_\alpha u_0(r)=\alpha r^{-\alpha-1}\int_0^r\Delta_\alpha u_0(s)s^\alpha\mathrm ds-\Delta_\alpha u_0(r)
\end{equation*}
and $\lim_{r\to0}u''_0(r)=-\Delta_\alpha u_0(0)/(\alpha+1)$. Also,
\begin{align*}
u_0'''(r)&=\alpha r^{-1}\Delta_\alpha u_0(r)-\alpha(\alpha+1)r^{-\alpha-2}\int_0^r\Delta_\alpha u_0(s)s^\alpha\mathrm ds-(\Delta_\alpha u_0)'(r)\\
&=\dfrac{\alpha(\alpha+1)}{r^{\alpha+2}}\int_0^r\left(\Delta_\alpha u_0(r)-\Delta_\alpha u_0(s)\right)s^\alpha\mathrm ds-(\Delta_\alpha u_0)'(r).
\end{align*}
Given $\varepsilon>0$, by $(\Delta_\alpha u_0)'(0)=0$ there exists $\delta>0$ such that $|(\Delta_\alpha u_0)'(r)|\leq\varepsilon$ and $|\Delta_\alpha u_0(r)-\Delta_\alpha u_0(s)|\leq\varepsilon(r-s)$ for all $0<s\leq r<\delta$. Thus
\begin{equation*}
|u_0'''(r)|\leq\varepsilon\dfrac{\alpha(\alpha+1)}{r^{\alpha+1}}\int_0^rs^\alpha\mathrm ds+\varepsilon=(\alpha+1)\varepsilon,\quad\forall r\in(0,\delta).
\end{equation*}
Therefore $u_0\in C^3([0,R])$ with $u_0'''(0)=0$.
\end{proof}
%\textcolor{blue}{Maximum Principle (Theorem 10.1 in Gilbarg-Trudinger) can be applied?}

\begin{proof}[Proof of Theorem \ref{theo4}]
    It is a direct consequence of Propostions \ref{prop42} and \ref{propclsol1}.
\end{proof}

\subsection{Application to Adams-Trudinger-Moser case}

Let us define the operator $\Delta_3 u:=-r^{-3}(r^3u')'$. Consider the following problem
\begin{equation}\label{problem2}
\left\{\begin{array}{ll}
\Delta_3 ^2u=r^{\theta-3}f(r,u)&\mbox{in }(0,R),\\
u=\Delta_3 u=0&\mbox{in }R,\\
u'=(\Delta_3 u)'=0&\mbox{in }0,
\end{array}\right.
\end{equation}
with $\theta>-1$ and $f\colon[0,R]\times\mathbb R\to\mathbb R$ continuous. Also, we are considering the following space
\begin{equation*}
X_{ATM}=X^{2,2}_R(-1,1,3)\cap X^{1,2}_{0,R}(-1,1).
\end{equation*}
\begin{cor}
The norm
\begin{equation*}
\|u\|_{\Delta_3}:=\left(\int_0^R|\Delta_3u|^2r^3\mathrm dr\right)^{\frac12}
\end{equation*}
is equivalent to $\|\cdot\|_{X^{2,2}_R}$ in $X_{ATM}$.
\end{cor}
\begin{proof}
Follows directly from Proposition \ref{propequivnormL}.
\end{proof}
From this Corollary we obtain
\begin{equation*}
m_\Delta:=\inf_{u\in X_{ATM}\backslash\{0\}}\dfrac{\|u\|_{\Delta_3}}{\|u\|_{X^{2,2}_R}}>0.
\end{equation*}
We suppose an exponential growth on $f$ given by
\begin{equation}\label{eqfTM}\tag{$f_{ATM}$}
|f(r,t)|\leq c_1e^{\mu (m_\Delta t)^2}\quad\forall r\in[0,R],t\in\mathbb R,
\end{equation}
for some $c_1>0$ and $\mu>0$. We say that $u\in X_{ATM}$ is a \textit{weak solution of} \eqref{problem2} if
\begin{equation*}
\int_0^R\Delta_3u\Delta_3vr^3\mathrm dr=\int_0^Rf(r,u)vr^\theta\mathrm dr\quad\forall v\in X_{ATM}.
\end{equation*}
The energy functional $J\colon X_{ATM}\to\mathbb R$ is given by
\begin{equation*}
J(u)=\dfrac12\int_0^R|\Delta_3u|^2r^3\mathrm dr-\int_0^RF(r,u)r^\theta\mathrm dr,
\end{equation*}
where $F(r,t)=\int_0^tf(r,s)\mathrm ds$. Note that \eqref{eqfTM} guarantees
\begin{equation}\label{eqFTM}\tag{$F_{ATM}$}
|F(r,t)|\leq c_2e^{\widetilde\mu (m_\Delta t)^2}\quad\forall r\in[0,R],t\in\mathbb R,
\end{equation}
for some $c_2>0$ and $\widetilde\mu>\mu$. Moreover, $J$ is well defined by \eqref{eqFTM} and Proposition \ref{prop50}.

\begin{prop}\label{prop45}
Suppose $t\mapsto f(r,t)$ is an odd function with $f(r,t)\geq0$ for all $t\geq0$. If $\mu<\theta+1$, then the following supremum
\begin{equation}\label{eqsup}
\sup_{\underset{u\in X_{ATM}}{\|u\|_{\Delta_3}=1}}\int_0^RF(r,u)r^\theta\mathrm dr
\end{equation}
is attained for some $u_0\in X_{ATM}$ with $\|u_0\|_{\Delta_3}=1$. Moreover, $u_0$ is a weak solution of
\begin{equation*}
\left\{\begin{array}{ll}
\Delta_3 ^2u=r^{\theta-3}\lambda f(r,u)&\mbox{in }(0,R),\\
u=\Delta_3 u=0&\mbox{in }R,\\
u'=(\Delta_3 u')=0&\mbox{in }0,
\end{array}\right.
\end{equation*}
where $\lambda=\int_0^Rf(r,u_0)u_0r^\theta\mathrm dr.$
\end{prop}
\begin{proof}
    Let $(u_n)$ a maximizing sequence of \eqref{eqsup} with $\|u_n\|_{\Delta_3}=1$. Up to subsequence, $u_n\rightharpoonup u_0$ for some $u_0\in X_{ATM}$. From Theorem \ref{theo32} we have $u_n\rightarrow u_0$ in $L^q_\theta$ for all $q\geq1$. Fix $q>1$ with $q\mu<(\theta+1)m^2$. Using Mean Value Theorem, \eqref{eqfTM} and $s\mapsto e^{q\mu s^2}$ is convex we obtain
    \begin{align*}
    \int_0^R&|F(r,u_n)-F(r,u_0)|r^\theta\mathrm dr=\int_0^R|f(r,tu_n+(1-t)u_0)||u_n-u_0|r^\theta\mathrm dr\\
    &\leq c_1\int_0^Re^{\mu m_\Delta^2|tu_n+(1-t)u_0|^2}|u_n-u_0|r^\theta\mathrm dr\\
    &\leq c_1\left(\int_0^Re^{q\mu m_\Delta^2|tu_n+(1-t)u_0|^2}r^\theta\mathrm dr\right)^{\frac1q}\|u_n-u_0\|_{L^q_\theta}\\
    &\leq c_1\left(\int_0^Re^{q\mu (m_\Delta u_n)^2}r^\theta\mathrm dr+\int_0^Re^{q\mu (m_\Delta u_0)^2}r^\theta\mathrm dr\right)^{\frac1q}\|u_n-u_0\|_{L^q_\theta}.
    \end{align*}
    Since $\|m_\Delta u_n\|_{X^{2,2}_R}\leq\|u_n\|_{\Delta_3}=1$ we can apply Proposition \ref{prop51} to get
    \begin{equation}\label{eq35}
    \int_0^RF(r,u_0)r^\theta\mathrm dr=\sup_{\underset{u\in X_{ATM}}{\|u\|_{\Delta_3}=1}}\int_0^RF(r,u)r^\theta\mathrm dr.
    \end{equation}
    
From $f(r,t)=-f(r,-t)\geq0$ for all $t\geq0$ we obtain $F(r,u)\geq0$ and $F(r,\gamma u)\geq F(r,u_0)$ for all $\gamma\geq1$. We can suppose $F(r,u)$ is not always zero. Then $u_0\in X_{ATM}\backslash\{0\}$ because \eqref{eq35} implies $\int_0^RF(r,u_0)r^\theta\mathrm dr>0$. We can assume $\|u_0\|_{\Delta_3}=1$. Otherwise, defining $\widetilde u_0=u_0/\|u_0\|_{\Delta_3}$  we have $\|\widetilde u_0\|_{\Delta_3}=1$ and
\begin{equation*}
\int_0^RF(r,\widetilde u_0)r^\theta\mathrm dr\geq\int_0^RF(r,u_0)r^\theta\mathrm dr=\sup_{\underset{u\in X_{ATM}}{\|u\|_{\Delta_3}=1}}\int_0^RF(r,u)r^\theta\mathrm dr.
\end{equation*}
Therefore, we obtain $u_0\in X_{ATM}$ such that attain the supremum with $\|u_0\|_{\Delta_3}=1$. Moreover, applying Lagrange Multipliers in \eqref{eq35} we conclude
\begin{equation*}
\int_0^R\Delta_3u_0\Delta_3vr^3\mathrm dr=\lambda\int_0^Rf(r,u_0)vr^\theta\mathrm dr\quad\forall v\in X_{ATM},
\end{equation*}
where $\lambda=\int_0^Rf(r,u_0)u_0r^\theta\mathrm dr$.
\end{proof}

\begin{prop}\label{propclassicsolution2}
Suppose $u_0$ is a weak solution of \eqref{problem2}. Then $u_0\in C^4(0,R]$, $\Delta_3 u_0\in C^2(0,R]$ and $\Delta_3 ^2u_0=r^{\theta-3}f(r,u_0)$ $\forall r\in(0,R]$. Moreover, $u_0(R)=\Delta_3 u_0(R)=0$.
\end{prop}
\begin{proof}
This proof follows similarly to the proof of Proposition \ref{propclassicsolution1} with $\alpha=3$ except we do not need to suppose $p\geq2$ because \eqref{eqfTM} and Proposition \ref{prop50} guarantee $f(r,u_0(\cdot))\in L^1_\theta$ for each $r\in[0,R]$.
\end{proof}

\begin{prop}\label{propclsol2}
Suppose $u_0$ is a weak solution of \eqref{problem2} with $\theta>2$. Then $u_0\in C^4((0,R])\cap C^3([0,R])$ is a classical solution of \eqref{problem2} with $\Delta_3 u_0\in C^2((0,R])\cap C^1([0,R])$ and $u'_0(0)=(\Delta_3 u_0)'(0)=u_0(R)=\Delta_3 u_0(R)=0$. Moreover, $u''_0(0)=-\Delta_3 u_0(0)/4$ and $u'''_0(0)=0$.
\end{prop}
\begin{proof}
This proof follows similarly to the proof of Proposition \ref{propclsol1} if we show that $u_0\in C^1([0,R])$. We claim
\begin{equation}\label{cl1}
\Delta_3^2u_0=r^{\theta-3}f(r,u_0)\in L^p_{4p-\theta-2},\quad\forall p>1.
\end{equation}
Indeed, using \eqref{eqfTM} we get
\begin{equation*}
\int_0^R|r^{\theta-3}f(r,u_0)|^pr^{4p-\theta-2}\mathrm dr\leq c_1\int_0^Re^{\mu p(m_\Delta u_0)^2}r^{(\theta+1)(p-1)-1}\mathrm dr.
\end{equation*}
Then Proposition \ref{prop50} guarantees \eqref{cl1}.

We obtain $\Delta_3u_0\in X^{2,p}_R(2p-\theta-2,3p-\theta-2,2p-\theta-2)$ from Lemma \ref{lemmafjs} in \eqref{cl1}. Again, from Lemma \ref{lemmafjs} we conclude $u_0\in X^{2,p}_R(-\theta-2,p-\theta-2,2p-\theta-2)$. Therefore, Morrey case of Theorem \ref{theo32} implies $u_0\in C^1([0,R])$.
\end{proof}
\begin{proof}[Proof of the Theorem \ref{theo5}]
    It is a direct consequence of Propositions \ref{prop45} and \ref{propclsol2}.
\end{proof}

%-----------------------------------------------------------------------------------------------------------------------
%-----------------------------------------------------------------------------------------------------------------------

%-----------------------------------------------------------------------------------------------------------------------
%-----------------------------------------------------------------------------------------------------------------------

\end{document}